\documentclass[11pt,a4paper,oneside,reqno]{amsart}

\usepackage{amsmath, amsthm, amsopn, amssymb}
\usepackage{mathtools, thmtools, thm-restate}
\usepackage[ruled,vlined]{algorithm2e}
\usepackage{bbm}
\usepackage{enumerate}
\usepackage{enumitem}
\usepackage{hyperref}
\usepackage{stmaryrd} 
\usepackage{xcolor}

\usepackage{tikz}
\usetikzlibrary{topaths,calc}
\usepackage{subcaption}

\usepackage{geometry}
\declaretheorem[name=Theorem,within=section]{thm}
\newtheorem{lemma}[thm]{Lemma}
\newtheorem{prop}[thm]{Proposition}
\newtheorem{cor}[thm]{Corollary}
\newtheorem{fact}[thm]{Fact}

\newtheorem{claim}[thm]{Claim}

\newtheorem{conj}[thm]{Conjecture}

\theoremstyle{definition}
\newtheorem*{remark}{Remark}
\newtheorem{dfn}[thm]{Definition}

\newtheorem{example}[thm]{Example}

\newcommand{\cA}{\mathcal{A}}
\newcommand{\cB}{\mathcal{B}}
\newcommand{\cC}{\mathcal{C}}
\newcommand{\cF}{\mathcal{F}}

\newcommand{\cP}{\mathcal{P}}
\newcommand{\cQ}{\mathcal{Q}}

\newcommand{\cS}{\mathcal{S}}

\renewcommand{\Pr}{\mathbb{P}}

\newcommand{\1}{\mathbbm{1}} 

\newcommand{\eps}{\varepsilon}
\newcommand{\Aux}{\mathrm{Aux}}

\newcommand{\Cbad}{\mathcal{C}_{\textrm{bad}}}

\newcommand{\Ram}{\mathrm{Ram}} 

\newcommand{\br}[1]{\llbracket{#1}\rrbracket}

\renewcommand{\le}{\leqslant}
\renewcommand{\ge}{\geqslant}

\title{The List-Ramsey threshold for Families of Graphs}

\author{Eden Kuperwasser}
\address{School of Mathematical Sciences, Tel Aviv University, Tel Aviv 6997801, Israel}
\email{kuperwasser@tauex.tau.ac.il}

\author{Wojciech Samotij}
\address{School of Mathematical Sciences, Tel Aviv University, Tel Aviv 6997801, Israel}
\email{samotij@tauex.tau.ac.il}

\thanks{This research was supported in part by the European Research Council Consolidator Grant 101044123 (RandomHypGra), by the Israel Science Foundation grant 2110/22, and by the grant 2019679 from the United States--Israel Binational Science Foundation (BSF) and the United States National Science Foundation (NSF)}

\begin{document}

\begin{abstract}
Given a family of graphs $\cF$ and an integer $r$, we say that a graph is $r$-Ramsey for $\cF$ if any $r$-colouring of its edges admits a monochromatic copy of a graph from $\cF$. The threshold for the classic Ramsey property, where $\cF$ consists of one graph, in the binomial random graph was located in the celebrated work of R\"odl and Ruci\'nski.

In this paper, we offer a twofold generalisation to the R\"odl--Ruci\'nski theorem.  First, we show that the list-colouring version of the property has the same threshold. Second, we extend this result to finite families $\cF$, where the threshold statements might also diverge. This also confirms further special cases of the Kohayakawa--Kreuter conjecture.
Along the way, we supply a short(-ish), self-contained proof of the $0$-statement of the R\"odl--Ruci\'nski theorem.
\end{abstract}

\maketitle

\setcounter{tocdepth}{1}
\tableofcontents

\section{Introduction}

Given an integer $r \ge 2$ and graphs $F$ and $G$, we say that $G$ is $r$-Ramsey for $F$ if any $r$-colouring of the edges of $G$ admits a monochromatic copy of $F$. We denote this property by $G \in \Ram(F,r)$. Ramsey's theorem~\cite{Ram30} states that, for all $r$ and $F$, the clique $K_n$ is $r$-Ramsey for $F$ whenever $n$ is sufficiently large.

One of the efforts of probabilistic combinatorics in the last few decades has been to understand the behaviour of such Ramsey properties in random structures. More concretely, one can ask how the probability that the binomial random graph $G_{n,p}$ is $r$-Ramsey for some graph $F$ changes as $p$ traverses the interval $[0,1]$.
In a major breakthrough, R\"odl and Ruci\'nski~\cite{RodRuc95} answered this question by finding thresholds for the Ramsey property for all graphs $F$ and every number of colours $r$. To state the theorem, first define the maximum $2$-density of a graph $F$ to be \[ m_2(F) \coloneqq  \max \left \{ \frac{e_{F'} - 1}{v_{F'} - 2} \colon \emptyset \neq F' \subseteq F, \; v_{F'} \ge 3 \right \} \cup \left \{ \frac{1}{2}\right \}. \]

\begin{thm}[R\"odl--Ruci\'nski~\cite{RodRuc95}]
\label{thm:rodl-rucinski}
Given $r \ge 2$ and a graph $F$, there are positive constants $c_0, c_1$ such that
  \[
    \lim_{n \to \infty} \Pr \left( G_{n,p} \in \Ram(F, r) \right) =
    \begin{cases}
      1, & p \ge c_1 \cdot n^{-1/m_2(F)}, \\
      0, & p \le c_0 \cdot n^{-1/m_2(F)},
    \end{cases}
  \]
  unless $F$ is a star forest or $r=2$ and $F$ is the path on four vertices\footnote{This last case, which was originally missed, was later noticed in~\cite{FriKri00}.}. 
\end{thm}

In the two exceptional cases  not covered by Theorem~\ref{thm:rodl-rucinski}, the corresponding statement requires some amendment.  This is because, in both cases, the threshold is determined by the appearance of a fixed graph $G$ that is $r$-Ramsey for our forbidden graph.  More precisely, the Ramsey threshold coincides with the threshold $\hat{p}_G = \hat{p}_G(n)$ for the property that $G \subseteq G_{n,p}$.

In some cases, Theorem~\ref{thm:rodl-rucinski} can be strengthened still to yield a sharp threshold. The following theorem is a combination of the works of Friedgut and Krivelevich~\cite{FriKri00} and Friedgut, Kuperwasser, Samotij, and Schacht~\cite{FriKupSamSch22}.

\begin{thm}[\cite{FriKri00,FriKupSamSch22}]
  Let $r \ge 2$ and let $F$ be a graph that is either a clique, a cycle, or a tree\footnote{The work in~\cite{FriKri00} establishes sharp thresholds for trees, and~\cite{FriKupSamSch22} proves the result for the family of \emph{collapsible} graphs, which includes cliques and cycles;  see~\cite{FriKupSamSch22} for the exact definition of collapsibility.} (excluding the exceptional cases from before). Then there are positive constants $c_0, c_1$ and a function $c_0 \le c(n) \le c_1$ such that, for every $\eps > 0$, we have 
 \[
    \lim_{n \to \infty} \Pr \left( G_{n,p} \in \Ram(F, r) \right) =
    \begin{cases}
      1, & p \ge (1+ \eps) \cdot c(n) \cdot n^{-1/m_2(F)}, \\
      0, & p \le (1 - \eps) \cdot c(n) \cdot n^{-1/m_2(F)}.
    \end{cases}
  \]
\end{thm}

We note here that, in the exceptional cases, the Ramsey property does not have a~sharp threshold.  Further, even though $G_{n,p}$ is a.a.s.\ not Ramsey when $p \ll \hat{p}_G$, the probability that $G_{n,p}$ contains the fixed Ramsey graph $G$ is still bounded from below by some positive constant for any $p = \Theta(\hat{p}_G)$. \\

In this paper, we will be interested in two extensions to the graph Ramsey property. The first is a generalisation of the property to families of graphs. The second is a~list-colouring variant that was recently introduced in~\cite{AloBucKalKupSza21}.

\begin{dfn}[Family-Ramsey]
 Given an integer $r \ge 2$ and a family $\cF$ of graphs, we say that a graph $G$ is $r$-Ramsey with respect to $\cF$ if any $r$-colouring of the edges of $G$ admits a monochromatic copy of one of the members of $\cF$. We denote this property as $G \in \Ram(\cF,r)$.
\end{dfn}

\begin{dfn}[List-Ramsey]
 Given an integer $r \ge 2$ and a family $\cF$ of graphs, we say that a graph $G$ is $r$-list-Ramsey with respect to $\cF$ if there is an assignment of lists of $r$ colours to the edges of $G$, not necessarily identical, such that any colouring from these lists admits a monochromatic copy of some $F \in \cF$. We denote this property by $G \in \Ram_\ell(\cF, r)$.
\end{dfn}

In~\cite[Section 7]{FriKupSamSch22}, it was shown that, for all graphs $F$ bar the exceptional cases, if $p = \Theta(n^{-1/m_2(F)})$, then any constant-sized graph that appears in $G_{n,p}$ with some positive probability, as $n$ tends to infinity, cannot be $2$-list-Ramsey for $F$. This sparks some hope that we can extend Theorem~\ref{thm:rodl-rucinski} to the list-Ramsey threshold. 

Such hope turned out to be true for other Ramsey-type properties, such as those corresponding to Schur's theorem~\cite{Sch17} on sums and to van der Waerden's theorem~\cite{vdW27} on arithmetic progressions. The usual thresholds for these properties had been established in the works of Graham, R\"odl, and Ruci\'nski~\cite{GraRodRuc96} and R\"odl and Ruci\'nski~\cite{RodRuc95, RodRuc97}. The list-colouring variants for these properties were considered in~\cite{FriKupSamSch22}, in the context of establishing sharp thresholds, where it was shown that the list-colouring threshold coincides with the usual threshold.

\begin{remark}
In fact, both of these results are immediate corollaries of \cite[Theorem~1.7]{FriKupSamSch22}, which gives sufficient conditions for locating the threshold for list-colourability of random induced subgraphs of `almost-linear' hypergraphs.
\end{remark}

Returning to graphs and to the results of this paper, we divide our statement for graph families into two cases. First, we treat the `generic' case, where we look at finite families of graphs that do not include forests. Second, we allow forests, which invites some of the exceptional behaviour that we encountered in the statement of the original random Ramsey theorem. Given a finite family of graphs $\cF$ define \[m_2(\cF) \coloneqq \min \left\{ m_2(F) \colon F \in \cF  \right\}.\]
Note that forbidding forests is the same as asserting that $m_2(\cF) > 1$.

\begin{thm}[Generic case]
  \label{thm:generic-case}
  Given an integer $r \ge 2$ and a finite family of graphs $\cF$ with $m_2(\cF) > 1$, there are positive constants $c_0, c_1$ such that
  \[
    \lim_{n \to \infty} \Pr \left( G_{n,p} \in \Ram_\ell(\cF, r) \right) =
    \begin{cases}
      1, & p \ge c_1 \cdot n^{-1/m_2(\cF)}, \\
      0, & p \le c_0 \cdot n^{-1/m_2(\cF)}.
    \end{cases}
  \]
\end{thm}

In other words, as long as we exclude forests, we have a semi-sharp threshold corresponding to the minimum $2$-density amongst the graphs in $\cF$.

\begin{remark}
  The finiteness assumption is necessary.  For example, if we just allow $\cF$ to contain all graphs with $2$-density at least $t$, then $G_{n,p}$ will become $r$-Ramsey for $\cF$ as soon as $p \gg n^{-1}$.  Indeed, if $p \gg n^{-1}$, then a.a.s.\ $e(G_{n,p}) \gg n$.  Further, since any $r$-colouring of $G_{n,p}$ admits a colour class $C$ taking up at least a $(1/r)$-fraction of the edges, the ratio $\frac{e_C - 1}{n-2}$ tends to infinity with $n$, which implies that $C \in \cF$.   
\end{remark}

Observe that, for every $r \ge 2$ and every nonempty family $\cF$, we have
\begin{equation}
  \label{eq:Ram-hierarchy}
  \Ram_\ell(\cF, 2) \supseteq \Ram_\ell(\cF, r) \supseteq \Ram(\cF, r) \supseteq \bigcup_{F \in \cF} \Ram(F,r).
\end{equation}
This means that the $1$-statement in Theorem~\ref{thm:generic-case} follows from the $1$-statement in Theorem~\ref{thm:rodl-rucinski}.  We remark here that Nenadov and Steger~\cite{NenSte16} show how to prove the latter  very quickly using the hypergraph container lemma and a `supersaturated' version of Ramsey's theorem.  Therefore, we only need to prove the $0$-statement in Theorem~\ref{thm:generic-case}, where we may also assume that $r = 2$, by the first inclusion in~\eqref{eq:Ram-hierarchy}.

\begin{thm}[$0$-statement]
  \label{thm:general}
  For every finite family $\cF$ of graphs with $m_2(\cF) > 1$, there exists a positive constant $c_0$ such that, for all $p \le c_0 \cdot n^{-1/m_2(\cF)}$,
  \[
    \lim_{n \to \infty} \Pr \left(G_{n,p} \in \Ram_\ell(\cF, 2) \right) = 0.
  \]
\end{thm}

In contrast, allowing forests into your life only invites chaos and strife. While we are able to handle infinite families of graphs in this case, we face a vast array of new exceptions.

The most drastic digression from Theorem~\ref{thm:general} is when our family contains a star forest. It is not hard to see that in this case the threshold corresponds to $G_{n,p}$ containing a vertex of a sufficiently high degree. This also means that the threshold is coarse.

\begin{thm}[Star forests]
\label{thm:star_forests}
 Let $r \ge 2$ be an integer and let $\cF$ be an family of graphs containing a star forest. Let $D$ be the smallest maximum degree of a star forest in $\cF$, and let $s \coloneqq r(D-1) + 1$. Then  \[
    \lim_{n \to \infty} \Pr \left( G_{n,p} \in \Ram_\ell(\cF, r) \right) =
    \begin{cases}
      1, & p \gg n^{-(1 + 1/s)}, \\
      0, & p \ll n^{-(1 + 1/s)}.
    \end{cases}
  \]
\end{thm}

When we exclude star forests, the threshold will stay at $n^{-1}$ (abiding the $2$-density).  The corresponding $1$-statement is a straightforward consequence of the fact that, for every integer $d$, every graph with average degree at least $2d$ contains a copy of every tree with $d$ edges.  Indeed, we can thus a.a.s.\ find a copy of any forest $F$ in the largest colour class of every $r$-colouring of $G_{n,p}$, provided that $p \ge c_1 \cdot n^{-1}$ for a large enough constant $c_1$. So again we only need to describe the $0$-statement. Unlike the generic case however, in some cases the $0$-statement will only hold if $p \ll n^{-1}$; that is, it may no longer be true when $p \le c \cdot n^{-1}$ for any given constant $c > 0$. These examples stem from interactions between two types of graphs that we define below.

\begin{dfn}
A \emph{broom} is a tree made from a path of length two by connecting to one of its endpoints an arbitrary number of leaves, called \emph{hairs} of the broom. We say that a graph $F$ is a \emph{$\cB$-graph} if every component of $F$ is a~subgraph of a broom.

We say that $F$ is a \emph{$\cC^*$-graph} if every component of $F$ is either a subgraph of an odd cycle or a~star.  
\end{dfn}

\begin{figure}[htb]
\centering
\begin{tikzpicture}[scale=0.5]

\foreach \i in {0,1,2} {
\draw[fill] (0,10 - 1.5*\i) circle (0.2);
}
\draw[line width=1.5pt] (0,10) -- (0,8.5) -- (0,7);

\foreach \i in {0,1,...,4} {
\draw[fill] (-2 + \i ,5.5) circle (0.2);
\draw[line width=1.5pt] (0,7) -- (-2 + \i, 5.5);
}

\end{tikzpicture}
\caption{A broom with five hairs.}
\end{figure}

\begin{thm}[Allowing forests: list-Ramsey]
  \label{thm:general-forests}
  Let $r \ge 2$ be an integer and let $\cF$ be an arbitrary family of graphs not including a star forest.  There exists a positive constant $c_0$ such that, for all $p \le c_0 \cdot n^{-1}$,
  \[
    \lim_{n \to \infty} \Pr(G_{n,p} \in \Ram_\ell(\cF, r)) = 0,
  \]
  unless $r = 2$ and $\cF$ contains both a $\cB$-graph and a $\cC^*$-graph. In the latter case, while $\Pr(G_{n,p} \in \Ram_\ell(\cF, r)) \to 0$ whenever $p \ll n^{-1}$, it is bounded away from zero for every $p = \Omega(n^{-1})$.
\end{thm}

Interestingly, there are families $\cF$ containing $\cB$-graphs and $\cC^*$-graphs for which the usual Ramsey threshold is still semi-sharp.  This behaviour is determined by the $2$-colourability of an auxiliary hypergraph $\Aux(\cF)$, defined in Section~\ref{sec:forests}.

\begin{thm}[Allowing forests: usual Ramsey]
  \label{thm:non-list-forests}
  Let $r \ge 2$ be an integer and let $\cF$ be an arbitrary family of graphs not including a star forest. There exists a positive constant $c_0$ such that, for all $p \le c_0 \cdot n^{-1}$,
  \[
    \lim_{n \to \infty} \Pr(G_{n,p} \in \Ram(\cF, r)) = 0,
  \]
  unless $r = 2$ and $\Aux(\cF)$ is not $2$-colourable. In the latter case, while $\Pr(G_{n,p} \in \Ram(\cF, r)) \to 0$ whenever $p \ll n^{-1}$, it is bounded away from zero for every $p = \Omega(n^{-1})$.
\end{thm}

\begin{example}
  Consider $\cF \coloneqq \{B, C_3 \cup C_5\}$ where $B$ is a broom with at least two hairs. This family contains a $\cB$-graph and a $\cC^{*}$-graph, so we will have a coarse threshold on the list-Ramsey front when $r=2$. However, we will see in Section~\ref{sec:forests} that the auxiliary hypergraph corresponding to this family is isomorphic to $K_2$, and is therefore $2$-colourable. As a result, in the usual Ramsey sense we will still get a semi-sharp threshold. 
\end{example}

\subsection{Corollaries and applications}

Specialising the above theorems to the case where $\cF$ is a singleton gives us the following strengthening of Theorem~\ref{thm:rodl-rucinski}.

\begin{cor}[The list-Ramsey threshold]
Given $r \ge 2$ and a graph $F$, there are positive constants $c_0, c_1$ such that
  \[
    \lim_{n \to \infty} \Pr \left( G_{n,p} \in \Ram_\ell(F, r) \right) =
    \begin{cases}
      1, & p \ge c_1 \cdot n^{-1/m_2(F)}, \\
      0, & p \le c_0 \cdot n^{-1/m_2(F)},
    \end{cases}
  \]
  unless $F$ is a star forest, or $r=2$ and $F$ is the path on four vertices.
\end{cor}

This follows from the fact that the only exceptions can come from graphs $F$ that are simultaneously $\cB$-graphs and $\cC^{*}$-graphs. The only such graphs are star forests and the path on four vertices.

Another corollary of this work concerns asymmetric Ramsey thresholds. For graphs $H,L$ with $m_2(H) \ge m_2(L)$, define their \emph{mixed $2$-density} as
\[
  m_2(H,L) \coloneqq \max \left \{ \frac{e_{H'}}{v_{H'} - 2 + 1/m_2(L)} \colon \emptyset \neq H' \subseteq H \right \}.
\]
Given graphs $F_1, \dotsc, F_r$, we say that a graph $G$ is Ramsey for $(F_1, \dotsc, F_r)$ if any $r$-colouring of the edges of $G$ admits, for some $i \in \br{r}$, a copy of the graph $F_i$ in the colour~$i$. We denote this by $G \in \Ram(F_1, \dotsc, F_r)$. Kohayakawa and Kreuter~\cite{KohKre97} conjectured that the threshold for this property is governed by the mixed $2$-density of the two densest graphs.

\begin{conj}[Kohayakawa--Kreuter~\cite{KohKre97}]
  \label{conj:KK}
  Let $r \ge 2$ and let $F_1, \dotsc, F_r$ be graphs with $m_2(F_1) \ge \dotsb \ge m_2(F_r)$ and $m_2(F_2) > 1$. Then there are positive constants $c_0$ and $c_1$ such that 
  \[
    \lim_{n \to \infty} \Pr \left( G_{n,p} \in \Ram(F_1, \dotsc, F_r) \right) =
    \begin{cases}
      1, & p \ge c_1 \cdot n^{-1/m_2(F_1, F_2)}, \\
      0, & p \le c_0 \cdot n^{-1/m_2(F_1, F_2)}.
    \end{cases}
  \]
\end{conj} 

The $1$-statement of this conjecture was proved by Mousset, Nenadov, and Samotij~\cite{MouNenSam20}, but the $0$-statement remains a formidable foe. The original work of Kohayakawa and Kreuter~\cite{KohKre97} gives a proof in the case where the two densest graphs are cycles.  Since then, the $0$-statement in Conjecture~\ref{conj:KK} has been proved in the case where both $F_1$ and $F_2$ are cliques~\cite{MarSkoSpoSte09} and where $F_1$ is a clique and $F_2$ is a cycle~\cite{LieMatMenSko20}. More recently, Hyde~\cite{Hyde23} provided a proof for most pairs of regular graphs.

Note that $m_2(F_2) \le m_2(F_1, F_2) \le m_2(F_1)$ so when $F_1$ and $F_2$ have the same $2$-density, it also coincides with their mixed $2$-density. As a result, we can prove the following additional special case of Conjecture~\ref{conj:KK}.

\begin{cor}
  The Kohayakawa--Kreuter conjecture holds whenever $m_2(F_1) = m_2(F_2)$.
\end{cor}

Indeed, if we can colour $G_{n,p}$ only using the first two colours without creating a copy of $F_1$ or $F_2$ in either colour, then $G_{n,p}$ is not Ramsey for $(F_1, \dotsc, F_k)$. Therefore, the $0$-statement for this case follows from the $0$-statement for the property $\Ram(\{F_1, F_2\}, 2)$.

\subsection{Organization of the paper}

We continue with Section~\ref{sec:outline} in which we outline the steps for proving Theorem~\ref{thm:general} and Theorem~\ref{thm:general-forests}. The proof of Theorem~\ref{thm:general}, the generic case, is completed in two steps, corresponding to Section~\ref{sec:prob_lemma} and Section~\ref{sec:det_lemma}. Finally, Section~\ref{sec:forests} provides the proofs for Theorems~\ref{thm:star_forests}, \ref{thm:general-forests}, and~\ref{thm:non-list-forests}, which describe the cases where forests are allowed.

Another feature of our work is a rather speedy proof of the original random Ramsey theorem of R\"odl and Ruci\'nski (i.e., Theorem~\ref{thm:rodl-rucinski}) that does not rely on the Nash-Williams theorem.  Because dealing with families adds certain complications for one of the stages of the proof, carried out in Section~\ref{sec:det_lemma}, we begin that section by first considering the singleton case. Consequently, the (semi-)interested reader could stop their perusal at page~\pageref{page:end-of-proof-of-RR}. 

\subsection{Acknowledgements} 
We thank the anonymous referee for their careful reading and helpful comments.

\section{Outline of the proof}
\label{sec:outline}

We start with a general discussion of our proof strategy.  Let $\cP$ be a monotone property of graphs and suppose we wish to show that $G_{n,p}$ avoids $\cP$ with high probability whenever $p \le c n^{-1/\alpha}$, for some positive constants $\alpha, c$.  
Suppose that $G$ is a fixed graph that has the property $\cP$.  If $m(G) \le \alpha$, then $G_{n,p} \supseteq G$, and thus $G_{n,p} \in \cP$ by monotonicity, with probability $\Omega(1)$.  As a result, to have any hope of proving our statement, we must be able to show the following:
\[
  \text{(D) Any $G$ with $m(G) \le \alpha$ avoids $\cP$}.
\]
This is a necessary condition and sometimes it is also sufficient.  Indeed, if $p \le c n^{-1}$ for $c$ sufficiently small, then $m(G_{n,p}) \le 1$ with high probability.  Actually, this will be the case for our statement whenever $m_2(\cF) = 1$, i.e., when we allow forests.  However, if $p \gg n^{-1}$, then, with high probability, the density of $G_{n,p}$ tends to infinity.

In the latter case, we require another structural statement about $G_{n,p}$.  One way to go about this is to find a family of graphs $\cQ$ that contains all the minimal elements of $\cP$ and such that the following holds with high probability:
\[
  \text{(P) Every graph $G \in \cQ$ with $G \subseteq G_{n,p}$ has $m(G) \le \alpha$.}
\]
We refer to this type of statement as a \emph{probabilistic lemma} and to the previous statement (D), which does not involve any randomness, as a \emph{deterministic lemma}.

Given these statements the result follows readily.  Indeed, if $G_{n,p} \in \cP$, then it must also contain some minimal element $G \in \cP$.  However, this would mean that $G \in \cQ$ and therefore, with high probability, that $m(G) \le \alpha$.  As a result, $G$ avoids $\cP$, a contradiction.

This is the usual approach for establishing $0$-statements for Ramsey thresholds, which was taken, for example, in the short proof of the R\"odl--Ruci\'nski theorem~\cite{RodRuc95} by Nenadov and Steger~\cite{NenSte16}.  This is also how we prove Theorems~\ref{thm:general-forests} and~\ref{thm:non-list-forests}.  For the generic case, we will take a slightly different approach (similar to the approach of Liebenau--Mattos--Mendon\c{c}a--Skokan~\cite{LieMatMenSko20}), which allows us to weaken the requirement of the probabilistic lemma at the cost of strengthening the deterministic lemma.  Indeed, instead of arguing about the maximum density of members of $\cQ$ we will only ask that, with high probability,
\[\text{(P$^-$) Every graph $G \in \cQ$ with $G \subseteq G_{n,p}$ has $e_G / v_G \le \alpha$.}\]
Subsequently, for the same proof to work, we need to show that the minimal elements of $\cP$ are themselves dense (and not just that they contain a dense subgraph):
\[\text{(D$^+$) Any minimal element $G \in \cP$ has $e_G / v_G > \alpha$.}\]

While the difference between (D) and (D$^+$) might look superficial, not being able to control the maximum density disallows some of the tools that were essential in the argument of R\"odl and Ruci\'nski (as well as that of Nenadov and Steger), such as the Nash-Williams theorem.  Instead, we build on the work from~\cite{FriKupSamSch22} and deploy a discharging argument that also allows us to deal with list-colouring.  One feature of this approach is that we get a proof that is self-contained and---due to the weakened probabilistic lemma---relatively short.  With more work, we could upgrade our proof of (P$^-$) to a proof of (P) by showing that, with high probability, every graph in $\cQ$ that appears in $G_{n,p}$ has a bounded number of vertices.  However, we do not pursue this here, as the weaker assertion (P$^-$) is sufficient for our purposes.

We now turn to applying this scheme in our context.  As we mentioned above, the case $m_2(\cF) = 1$ only requires~(D) and the $0$-statement holds if and only if (D) holds.  Since (D) fails for a larger class of families in the context of list-Ramsey property, as opposed to the classical Ramsey property, we obtain different threshold statements.  This is discussed fully in Section~\ref{sec:forests}.

Next, let $\cF$ be a finite family with $m_2(\cF) > 1$.  We may assume that $\cF$ consists of strictly $2$-balanced graphs.  (Recall that a graph $F$ is \emph{strictly $2$-balanced} if $m_2(F) > m_2(F')$ for every nonempty subgraph $F' \subseteq F$.)  Indeed, we may define $\cF'$ by going over every $F \in \cF$ and adding to $\cF'$ a strictly $2$-balanced subgraph $F' \subseteq F$ with the same $2$-density as $F$; such $F'$ exists since $m_2(F) > 1$.  By construction, $m_2(\cF) = m_2(\cF')$.  Moreover, any colouring of $G_{n,p}$ that avoids all graphs in $\cF'$ also avoids all graphs in $\cF$.  We may therefore replace $\cF$ and prove the result for $\cF'$ in its stead.

\begin{dfn}
  Given a graph $G$ and a family of graphs $\cF$, define the \emph{$\cF$-hypergraph of $G$} to be the hypergraph on the vertex set $E(G)$ whose hyperedges correspond to copies of $F$ in $G$, for all $F \in \cF$.
\end{dfn}

\begin{dfn}
  A graph $C$ is called an \emph{$\cF$-cluster} if the $\cF$-hypergraph of $C$ is connected.  Let $\cC_\cF$ be the family of all $\cF$-clusters.
\end{dfn}

The family $\cQ$ that we are going to choose in (P$^-$) is the family of all $\cF$-clusters;  it is not hard to see that a minimal Ramsey graph for $\cF$ must be an $\cF$-cluster.  This explicit description of $\cQ$ will allow us to prove a probabilistic lemma without much effort.

To summarise, the following two lemmas give us Theorem~\ref{thm:general}.

\begin{lemma}[The Probabilistic Lemma]
  \label{lem:gen-probabilistic}
  Let $\cF$ be a finite family of strictly $2$-balanced graphs.  There exists a positive constant $c$ such that, if $p \le cn^{-1/t}$, then a.a.s.\ every $\cF$-cluster $C \subseteq G_{n,p}$ satisfies $e_C / v_C \le t$.
\end{lemma}

\begin{lemma}[The Deterministic Lemma]
  \label{lem:gen-deterministic}
  Let $\cF$ be a family of strictly $2$-balanced graphs, each containing a cycle.  Every graph $G$ that is minimally list-Ramsey with respect to $\cF$ satisfies $e_G/v_G > m_2(\cF)$.
\end{lemma}

The proofs of these lemmas are presented in the next two sections.  While we prove the Probabilistic Lemma straight away, we will take a more cautious approach with the Deterministic Lemma.  Since our argument is somewhat technical, we first treat the case where $\cF$ is a singleton\footnote{This also allows us to finish the proof of the $0$-statement in the R\"odl--Ruci\'nski theorem earlier.}.  The proof for single graphs already gives a good idea of how the general argument works and lays much of the groundwork towards it.

\section{Proof of the Probabilistic Lemma}
\label{sec:prob_lemma}

Let $\Cbad$ be the collection of all $C \in \cC_{\cF}$ such that  $e_C / v_C > t$.  We will show that there exist positive constants $\eps$ and $L$ and a family $\cS$ of subgraphs of $K_n$ with the following properties:
\begin{enumerate}[label=(\alph*)]
\item
  \label{item:sig-1}
  Every set in $\Cbad$ contains some element of $\cS$.
\item
  \label{item:sig-2}
  Every $S \in \cS$ satisfies $e_S \ge t \cdot v_S + \eps$, or both $v_S \ge \log n$ and $e_S \ge t \cdot (v_S - 2)$.
\item
  \label{item:sig-3}
  For every $k$, there are at most $(Ln)^k$ graphs $C \in \cS$ with $v_C = k$.
\end{enumerate}

We first show that the existence of such a collection $\cS$ implies the assertion of the lemma.  Indeed, it follows from~\ref{item:sig-1} and the union bound that
\[
  \Pr(\text{$C \subseteq G_{n,p}$ for some $C \in \Cbad$}) \le \Pr(\text{$S \subseteq G_{n,p}$ for some $S \in \cS$}) \le \sum_{S \in \cS} p^{e_S}.
\]
Moreover, by~\ref{item:sig-2} and~\ref{item:sig-3},
\[
  \begin{split}
    \sum_{S \in \cS} p^{e_S} & \le \sum_{k} (Ln)^k \cdot \left(p^{tk+\eps} + \1_{k \ge \log n} \cdot p^{t(k-2)}\right) \\
    & \le \sum_{k} (c^tL)^k \cdot p^{\eps} + \sum_{k \ge \log n} (c^t L)^{k} \cdot p^{-2t}.
  \end{split}
\]
We now choose $c$ sufficiently small so that $c^t L \le \min\{e^{-4t}, 1/2\}$. This way, the first sum above is at most $p^{\eps}$ and the second sum is at most $2 \cdot (n^2p)^{-2t}$, so both tend to zero.

To find the family $\cS$, we will first define an exploration process of the clusters. We first fix a labeling of the vertices of $K_n$, which induces an ordering of all subgraphs according to the lexicographical order. Now, given $C \in \cC_{\cF}$, we start with $C_0 \coloneqq \{e_0\}$, where $e_0$ is the smallest edge of $C$.  As long as $C_i \neq C$, do the following: Since $C \neq C_i$ and $C$ is an $\cF$-cluster, there must be a copy of some $F \in \cF$ in $C$ that intersects $C_i$ but is not fully contained in $C_i$.  Call such a copy \emph{regular} if it has $2$-density exactly $t$ and intersects $C_i$ in exactly one edge $e$, called its \emph{root}; otherwise, call the copy \emph{degenerate}.  We form $C_{i+1}$ from $C_i$ by adding to it one such overlapping copy of $F$ as follows: If there is a regular copy of some $F \in \cF$, add the smallest one among the copies rooted at the edge that arrived to $C$ the earliest among all edges of $C_i$; otherwise, add the smallest degenerate copy of $F$.  Finally, given an integer $\Gamma$ let
\[
  \tau = \tau(C) \coloneqq \min\{i : \text{$i$ is the $\Gamma$-th degenerate step or $v_{C_i} \ge \log n$ or $C_i = C$}\}
\]
and let
\[
  \cS \coloneqq \{C_{\tau(C)} : C \in \Cbad\}.
\]
This definition guarantees that $\cS$ satisfies~\ref{item:sig-1} above; we will show that, for sufficiently large $\Gamma$ and $L$ and a sufficiently small $\eps$, it also satisfies~\ref{item:sig-2} and~\ref{item:sig-3}.

\begin{claim}
  There exist $\Gamma = \Gamma(\cF) > 0$ and $\eps = \eps(\cF) > 0$ such that item~\ref{item:sig-2} holds.
\end{claim}

\begin{proof}
  We first argue that~\ref{item:sig-2} holds in the case where $S \in \Cbad$.  To this end, note that $t$ is a rational number and therefore there is an integer $b$ that depends only on $\cF$ such that $bt$ is an integer.  Since both $b e_S$ and $b t v_S$ are integers and $e_S > t v_S$, as $S \in \Cbad$, we must have $be_S \ge btv_S + 1$, which means that~\ref{item:sig-2} holds with $\eps = 1/b$.

  Consider an arbitrary $S \in \cS \setminus \Cbad$ and let $C \in \Cbad$ be a cluster such that $S = C_{\tau(C)} \neq C$.  Let $d_i$ denote the number of degenerate steps taken when constructing $C_i$. It suffices to show that there exists a positive $\eta = \eta(\cF)$ such that
  \begin{equation}
    \label{eq:degenerate-steps-balance}
    (e_{C_i} - 1) - t \cdot (v_{C_i} - 2) \ge \eta \cdot d_i.
  \end{equation}
  Indeed, set $\Gamma \coloneqq \lceil 2t/\eta \rceil$.  Inequality~\eqref{eq:degenerate-steps-balance} with $i = \tau(C)$ may be rewritten as $e_S - t \cdot v_S \ge 1 - 2t + \eta \cdot d_{\tau(C)}$.  Consequently, if $d_{\tau(C)} = \Gamma$, then $e_S \ge t \cdot v_S + 1 - 2t + \eta \cdot \Gamma \ge t \cdot v_S + 1$.  Otherwise, if $d_{\tau(C)}< \Gamma$, we only know that $e_S \ge t \cdot v_S - 2t+1$, but it must be that $v_S \ge \log n$, by the definition of $\tau$ and since we assumed that $S \neq C$.
  
  We now turn to proving~\eqref{eq:degenerate-steps-balance}; we do so by induction on $i$. Note that $C_0$ is an edge and therefore
  \[
    (e_{C_0} - 1) - t \cdot (v_{C_0} - 2) = 0.
  \]
  From there on, we wish to show that the left-hand side of~\eqref{eq:degenerate-steps-balance} does not change with any regular step, while it grows by at least $\eta$ with any degenerate step.
  
  Indeed, with every regular step we add a copy of $F \in \cF$ without one edge, meaning that $e_{C_i} = e_{C_{i-1}} + e_F - 1$ and $v_{C_i} = v_{C_{i-1}} + v_F - 2$, where, by the regularity condition, $(e_F -1) - t \cdot (v_F - 2) = 0$. Therefore, a regular step will not change the left-hand side. On the other hand, with any degenerate step we have two options. In our extending step, we used a graph $F \in \cF$ which either intersected $C_{i-1}$ in one edge, but had $2$-density strictly larger than $t$, or it intersected $C_{i-1}$ in at least two edges. In the first case, we know that $(e_F - 1) - t\cdot (v_F-2) > 0$. Since our family $\cF$ is finite, we may define $\eta_1$ to be the minimum of $(e_F - 1) - t\cdot (v_F-2)$ over all $F \in \cF$ with $m_2(F) > t$. For the second case, if the intersection is in a strict subgraph $F' \subsetneq F$ with $e_{F'} \ge 2$, then the left-hand side grows by $(e_F - e_{F'}) - t\cdot (v_F - v_{F'})$. Using the fact that $F$ is strictly $2$-balanced and $2 \le e_{F'} < e_F$, we have $(e_{F'}-1)/(v_{F'}-2) < (e_F-1)/(v_F-2)$ and, consequently, \[t \le \frac{e_F - 1}{v_F - 2} = \frac{(e_F - e_{F'}) + (e_{F'} - 1) }{(v_F - v_{F'}) + (v_{F'} - 2)}  < \frac{e_F - e_{F'}}{v_F - v_{F'}}.\] So $(e_F - e_{F'}) - t\cdot (v_F - v_{F'}) > 0$, and again we can take $\eta_2$ to be the minimum of this quantity over all possible $F$ and $F'$. Therefore, \eqref{eq:degenerate-steps-balance} holds with $\eta = \min\{\eta_1, \eta_2\}$.
\end{proof}

\begin{claim}\label{clm:components_number}
  There exists a constant $L = L(\cF)$ such that item~\ref{item:sig-3} holds.
\end{claim}
\begin{proof}
  Suppose that $S \in \cS$ has $k$ vertices.  Let $C \in \Cbad$ be such that $S= C_{\tau(C)}$ and consider the exploration process on $C$.  If the $i$-th step is degenerate, it can be uniquely described by specifying the graph $F \in \cF$, the intersection $F'$ of $C_{i-1}$ with $F$, and the sequence of $v_F - v_{F'}$ vertices of $K_n$ that complete it to this copy of $F$. If the $i$-th step is regular, it can be uniquely described by its root in $C_{i-1}$, the graph $F \in \cF$, the edge of $F$ that corresponds to the root, and the sequence of $v_F - 2$ vertices of $K_n$ that complete it to a copy of $F$.

  There are at most $n^k$ ways to choose the ordered sequence of vertices that were added by the exploration process.  Since each regular step adds at least one new vertex to the cluster, there are at most $k$ regular steps.  Further, since the number of degenerate steps is at most $\Gamma$, we have $\tau(C) \le k+\Gamma$.  In particular, there are at most $(k+\Gamma) \cdot 2^{k+\Gamma}$ ways to choose $\tau(C)$ and designate which of the $\tau(C)$ steps were regular and which were degenerate.  For every degenerate step, there are at most \[\sum_{F \in \cF} \sum _{\ell=2}^{v_F} \binom{v_F}{\ell} \cdot k^\ell \le |\cF| \cdot (k+1)^{M_v} \] ways to choose $F \in \cF$ and its intersection with $S$, where $M_v \coloneqq \max \{v_F \colon F \in \cF \}$.  As for the regular steps, since: the root at the $i$-th step is the edge that was added to $C_{i-1}$ the earliest among all those that are a part of a regular copy; and a copy that is regular at some step $j$ is regular at every earlier step $i$ as long as its root belongs to $C_{i-1}$, the sequence of `birth times' of all roots in the exploration process is non-decreasing.  In particular, the number of ways to choose the sequence of roots for all the regular steps is at most the number of non-decreasing $\br{e_S}$-valued sequences of length at most $k$, which is at most $\binom{e_S+k}{k}$.  Finally, since every step increases the number of edges of $C_i$ by at most $M_e \coloneqq \max \{ e_F \colon F \in \cF \}$, we have
  \[ e_S \le 1 + \tau(C) \cdot M_e \le 1 + (k + \Gamma) \cdot M_e. \]
  To summarise,
  \[
    |\cS| \le n^k \cdot (k+\Gamma) \cdot 2^{k+\Gamma} \cdot \left( |\cF|  \cdot (k+1)^{M_v} \right)^\Gamma \cdot \binom{(k+\Gamma) \cdot M_e + k + 1}{k} \cdot \left(|\cF| \cdot M_e\right)^k,
  \]
  which is at most $(Ln)^k$, provided that $L$ is sufficiently large.
\end{proof}
Since we have already shown how the existence of a collection $\cS$ satisfying items~\ref{item:sig-1}--\ref{item:sig-3} implies the assertion of the lemma, the proof is now complete.

\section{Proof of the Deterministic Lemma}
\label{sec:det_lemma}

Suppose that a graph $G$ is minimally list-Ramsey with respect to $\cF$.  It is not hard to see that $G$ must be an $\cF$-cluster.  We further claim that $G$ must have a certain `rigidity' property that is slightly stronger than the property that every edge of $G$ is the sole intersection of some two copies of graphs from~$\cF$.

\begin{dfn}
  A connected hypergraph $H$ is a \emph{core} if, for every edge $e \in H$ and vertex $v \in e$, there is an edge $e' \in H$ such that $e \cap e' = \{v\}$.
\end{dfn}

\begin{fact}
  Every minimally non-$2$-list-colourable hypergraph is a core.
\end{fact}

\begin{proof}
  Let $H$ be a non-core hypergraph whose every proper subhypergraph is $2$-list-colourable.  We will show that $H$ is also $2$-list-colourable.  Indeed, suppose we assign a list of two colours to every vertex of $H$.  Since it is not a core, there is an edge $e \in H$ and a vertex $v \in e$ such that no edge $e' \in H$ intersects $e$ solely in $v$.  Since the hypergraph $H \setminus e$ is $2$-list-colourable, we may properly colour its vertices from their lists.  If $e$ is not monochromatic under this choice of colours, then we have coloured $H$.  Otherwise, all vertices of $e$ were assigned the same colour. In this case, switch the colour of $v$ to the other option in its list.  We claim that this gives a proper colouring of $H$.  Indeed, $e$ is no longer monochromatic; moreover, if this change of colour affected some edge $e'$, then it must contain $v$.  However, if this is the case, the by our assumption $e'$ must intersect $e$ in at least one more vertex $v'$.  But since the colour of $v$ was changed just so it is different than the colour of all other vertices in $e$, then now $v$ and $v'$ are coloured distinctly, and therefore $e'$ cannot be monochromatic.
\end{proof}

We are now ready to define our rigidity property.  We will say that a graph $G$ is an \emph{$\cF$-core} if the $\cF$-hypergraph of $G$ contains a spanning, connected subhypergraph that is a core\footnote{This does not mean that the $\cF$-hypergraph of $G$ itself is a core.  For example, the $\{K_4, K_5\}$-hypergraph of $K_6$ is not a core. However, the subhypergraph given by the $K_4$-hyperedges is a core, making $K_6$ a $\{K_4, K_5\}$-core.}.  If $\cF$ does not contain any graphs with $2$-density at most two, the assertion of the Deterministic Lemma will hold under the weaker assumption that $G$ is a connected $\cF$-core\footnote{Since every strictly $2$-balanced graph is connected, all minimally list-Ramsey graphs for a family of strictly $2$-balanced graphs are connected as well.}, with the single exceptional case $G = K_6$, which is treated in Claim~\ref{claim:K6-not-list-Ramsey} below.

\begin{prop}
  \label{prop:cores-family}
  Suppose that $\cF$ is a family of strictly $2$-balanced graphs with
  $m_2(\cF) > 2$.  If a connected $\cF$-core $G$ satisfies $e_G/v_G \le
  m_2(\cF)$, then $G=K_6$.
\end{prop}

As there do exist $\cF$-cores $G$ with $e_G/v_G \le m_2(\cF)$ for many families $\cF$ with $m_2(\cF) \le 2$, this case will require a different argument that uses the full power of the assumption that $G$ is minimally list-Ramsey.

\begin{prop}
  \label{prop:eGvG-more-than-two}
  If $G$ is minimally list-Ramsey for some family $\cF$ of strictly $2$-balanced graphs,
  each containing a cycle, then $e_G/v_G > \min\{m_2(\cF), 2\}$.
\end{prop}

Finally, for the sake of providing a short, self-contained proof of the $0$-statement in the R\"odl--Ruci\'nski theorem (Theorem~\ref{thm:rodl-rucinski}), we will separately treat the case where $\cF$ comprises only one graph.  For brevity, we write $F$-core in place of $\{F\}$-core.  The following single-graph analogue of Proposition~\ref{prop:cores-family} characterises all $F$-cores $G$ with $e_G/v_G \le m_2(F)$ for all strictly $2$-balanced graphs $F$ that contain a cycle, except the triangle.  Clearly, $K_6$ is not Ramsey for $K_4$ and it is not difficult to prove that every $3$-regular graph is a union of two forests (Claim~\ref{claim:3-regular-2-colouring} below is a generalisation of this statement); showing that no graph $G$ with $e_G/v_G \le m_2(K_3)$ is minimally Ramsey for $K_3$ is only a little more difficult, see Claim~\ref{claim:K5} below.

\begin{prop}
  \label{prop:cores-single-graph}
  Suppose that $F$ is a strictly $2$-balanced graph with $m_2(F) > 1$.  If a~connected $F$-core $G$ satisfies $e_G/v_G \le m_2(F)$, then either:
  \begin{enumerate}
  \item
    $F = K_3$,
  \item
    $F = K_4$ and $G = K_6$, or
  \item
    $F = C_4$ and $G$ is $3$-regular.
  \end{enumerate} 
\end{prop}

The remainder of this section is organised as follows.  In Section~\ref{sec:some-tools}, we establish two elementary graph-theoretic lemmas that we use repeatedly while proving Propositions~\ref{prop:cores-family}--\ref{prop:cores-single-graph}. These lemmas appeared in \cite[Section 7]{FriKupSamSch22}, but we include the proofs for the sake of being self-contained.  In Section~\ref{sec:single-graph-case}, we treat the single-graph case of the Deterministic Lemma, proving Proposition~\ref{prop:cores-single-graph}, showing that $K_6$ is not list-Ramsey for $K_4$, that no $3$-regular graph is list-Ramsey for $C_4$, and that every graph $G$ that is minimally list-Ramsey for $K_3$ satisfies $e_G/v_G > m_2(K_3)$. Some of the arguments for the single-graph case already appear in~\cite{FriKupSamSch22}. In Section~\ref{sec:general-families}, we treat the general case of arbitrary (finite) families $\cF$.  We first prove the easier Proposition~\ref{prop:cores-family} and argue that $K_6$ is not list-Ramsey for any family $\cF$ with $m_2(\cF) \ge e_{K_6} / v_{K_6}$.  Second, we prove the more intricate Proposition~\ref{prop:eGvG-more-than-two}.  We feel that it is worth mentioning that Section~\ref{sec:general-families} reuses several claims established in the earlier subsection and thus our proof of the single-graph case of the Deterministic Lemma may be viewed as a warm-up for the general case.

\subsection{Some tools}
\label{sec:some-tools}

Given a graph $F$ and a set $W \subseteq V(F)$, it will be convenient to denote by $\bar{e}_F(W)$ the number of edges incident with a vertex of $W$, i.e., $\bar{e}_F(W) \coloneqq e_F - e_{F - W}$, where $F-W$ denotes the subgraph of $F$ induced by the set $V(F) \setminus W$.

\begin{lemma}[Helpful Lemma, \cite{FriKupSamSch22}]
  Suppose that $F$ is strictly $2$-balanced and that $W \subseteq V(F)$ satisfies $1 \le |W| \le v_F-3$. Then
  \[
    \bar{e}_F(W) > m_2(F)\cdot |W|.
  \]
\end{lemma}

\begin{proof}
  Since $F-W$ is a proper subgraph of $F$ with at least three vertices and $F$ is strictly $2$-balanced,
  \[
    \frac{e_F - 1 - \bar{e}_F(W)}{v_F-2-|W|} = \frac{e_{F-W}-1}{v_{F-W}-2} < m_2(F) = \frac{e_F-1}{v_F-2},
  \]
  which means that $(e_F-1)\cdot |W|<(v_F-2) \cdot \bar{e}_F(W)$ and the result follows. 
\end{proof}

\begin{lemma}[Discharging, \cite{FriKupSamSch22}]
\label{lem:discharging}
Suppose that  $e_G/v_G \le k + \eps$, where $k \ge 1$ is an integer and $\eps \in [0,1)$. One of the following holds:
\begin{enumerate}[label=(D\alph*)]
    \item
      \label{item:discharging-1}
      $G$ has a vertex of degree at most $2k$,
    \item
      \label{item:discharging-2}
      $G$ has a vertex of degree $2k+1$ with a neighbour of degree at most $2k+2$ and $\eps \ge 1/2$, or
    \item
      \label{item:discharging-3}
      $G$ has a vertex of degree $2k+3$ with two neighbours of degree $2k+1$ and $\eps \ge 7/8$.
\end{enumerate}
\end{lemma}

\begin{proof}
  Since $\delta(G) \le 2e_G/v_G \le 2k+2\eps < 2k+2$, we have $\delta(G) \le 2k+1$ and, if $\eps < 1/2$, then $\delta(G) \le 2k$.  We may thus assume that $\delta(G) = 2k+1$ and $\eps \ge 1/2$, since otherwise~\ref{item:discharging-1} holds.  We may further assume that all neighbours of every vertex of degree $2k+1$ have degrees at least $2k+3$, since otherwise~\ref{item:discharging-2} holds.

  Assign to each $v \in V(G)$ a charge of $\deg(v) - 2(k+\eps)$ and note that the average charge is non-positive.  We define the following discharging rule: every vertex of degree $2k+1$ takes a charge of $\frac{2\eps-1}{2k+1}$ from each of its neighbours.  By our assumption, no vertex of degree $2k+1$ or $2k+2$ sends charge to any of its neighbours.  In particular, the final charge of a vertex of degree $2k+1$ is
  \[
    2k+1-2(k+\eps)+(2k+1) \cdot \frac{2\eps-1}{2k+1} = 0
  \]
  and the final charge of a vertex of degree $2k+2$ is $2k+2 - 2(k+\eps) > 0$.

  Since the total charge remains unchanged, the final charge of some vertex of degree at least $2k+3$ must be non-positive.  Let $v$ be one such vertex.  Suppose that $\deg(v) = 2k+t$, where $t \ge 3$, and that $v$ has $x$ neighbours with degree $2k+1$.  Since the final charge of $v$ is
  \[
    2k+t - 2(k+\eps) - x \cdot \frac{2\eps - 1}{2k+1} \le 0,
  \]
  we have
  \[
    (t-2)(2k+1) < \frac{t - 2\eps}{2\eps - 1} \cdot (2k+1) \le x \le 2k+t,
  \]
  which implies that $t < 3 + 1/k \le 4$.  Therefore, $t=3$ and $x > 2k+1 \ge 3$, which means that some vertex of degree $2k+3$ has more than three neighbours of degree $2k+1$.  Moreover, we also have $\frac{3-2\eps}{2\eps-1} \cdot (2k+1) \le 2k+3$, which implies that $\eps \ge \frac{4k+3}{4k+4} \ge \frac{7}{8}$.
\end{proof}

\subsection{The single-graph case}
\label{sec:single-graph-case}

We begin by proving Proposition~\ref{prop:cores-single-graph} in the case where $F$ is not one of the three exceptional graphs.  Note that these are the only strictly $2$-balanced graphs with at most four vertices, so disregarding them allows us to apply the Helpful Lemma to $F$ with $W$ of size one or two.

\begin{proof}[Proof of Proposition~\ref{prop:cores-single-graph} in the case $v_F \ge 5$]
  Our goal is to prove that no graph $G$ with $e_G / v_G \le m_2(F)$ is an $F$-core.  Suppose by contradiction that $G$ is such a graph and write $m_2(F) = k + \eps$, where $k \ge 1$ is an integer and $\eps \in [0,1)$.  We split the argument into three cases, depending on which item of Lemma~\ref{lem:discharging} holds.

  \medskip
  \noindent
  \textit{Case 1.}
  Suppose that $G$ has a vertex $v$ of degree at most $2k$ and let $e$ be an edge incident to $v$.  By assumption, there are two copies of $F$ in $G$ whose sole intersection is the edge $e$ and thus $\deg(v) \ge 2\delta(F) - 1$.  However, since $v_F > 3$, we may apply the Helpful Lemma to learn that $\delta(F) > m_2(F) \ge k$, which implies that $\deg(v) \ge 2k+1$, a contradiction.

  \medskip
  \noindent
  \textit{Case 2.}
  Suppose now that $G$ has an edge $uv$ with $\deg(u) = 2k+1$ and $\deg(v) \le 2k+2$, and that $\eps \ge 1/2$.  By our assumption, there are two copies of $F$ in $G$ that intersect only in $uv$.  Since $v_F > 4$, we may use the Helpful Lemma to learn that the endpoints of every edge of $F$ touch strictly more than $2(k+ \eps) \ge 2k+1$ edges.  Because the copies of $F$ intersect only in $uv$, this means that $\{u,v\}$ must touch at least $2(2k+2) - 1 = 4k + 3$ edges in $G$.  However, our assumption implies that $\{u, v\}$ touches at most $\deg(u) + \deg(v) - 1 \le 4k+2$ edges, a contradiction.

  \medskip

  \noindent
  \textit{Case 3.}
  Finally, suppose that $G$ has a vertex $v$ of degree $2k+3$ connected to two vertices $u_1, u_2$ of degrees $2k+1$, and that $\eps \ge \frac{7}{8}$. We may assume that $u_1 u_2 \notin G$, as otherwise we are in Case 2.  By our assumption, we find two copies of $F$ that intersect only in $u_1v$.  If none of them use $u_2v$, then we may argue exactly as in Case 2, as the edge $u_1v$ is the sole intersection of two copies of $F$ in the graph $G' \coloneqq G \setminus u_2v$ and $\deg_{G'}(u_1) = 2k+1$ and $\deg_{G'}(v) =2k+2$.

  We may thus assume that one of the two copies, say $F_1$, contains $u_2v$ while the other, $F_2$, does not.  Our assumption that $G$ is a core supplies another copy of $F$, say $F_3$, that intersects $F_1$ precisely in the edge $u_2v$.  Similarly as before, we will show that there are not enough edges in $G$ that touch $\{v, u_1, u_2\}$ to support these three copies of $F$.  Since~$\eps$, which is the fractional value of $m_2(F)$, is at least $7/8$, the graph $F$ must have at least ten vertices, and we are justified in applying the Helpful Lemma to sets $W$ of up to three vertices. In particular, we may conclude that $F_1$ contains at least $\left\lceil 3(k+\eps) \right\rceil = 3k + 3$ edges touching $\{v, u_1, u_2\}$, that $F_2$ contains at least $\left\lceil 2(k+\eps) \right\rceil = 2k+2$ edges touching $\{v, u_1\}$, and that $F_3$ contains at least $\left\lceil k+\eps \right\rceil = k+1$ edges touching $u_2$.  Since, among the edges mentioned, only $u_1v$ and $u_2v$ were counted twice, the set $\{v, u_1, u_2\}$ must touch at least $6k+4$ edges of $G$. However, the bounds on the degrees of $v$, $u_1$, and $u_2$ imply that they touch only $\deg(u_1) + \deg(u_2) + \deg(v) - 2 = 6k+3$ edges.
\end{proof}

In order to complete the proof of Proposition~\ref{prop:cores-single-graph}, we are now going to identify the $F$-cores in the case where $F \in \{K_4, C_4\}$.  Further, to conclude the proof of the Deterministic Lemma in the case $\cF = \{F\}$, we will have to show how to colour each $K_4$-core and $C_4$-core and prove that every graph $G$ that is minimally list-Ramsey with respect to $K_3$ satisfies $e_G/v_G > 2 = m_2(K_3)$.

\begin{claim}
  \label{claim:K4-core-is-K6}
  The only connected $K_4$-core $G$ with $e_G/v_G \le m_2(K_4)$ is $G = K_6$.
\end{claim}

\begin{proof}
  Suppose that $G$ is a connected $K_4$-core with $e_G/v_G \le m_2(K_4) = 5/2$.  Since $G$ must satisfy $\delta(G) \ge 2\delta(K_4)-1 = 5 \ge 2e_G/v_G$, see Case 1 in the proof of Proposition~\ref{prop:cores-single-graph}, $G$ must be $5$-regular.  Now, let $a_1 a_2 \in G$ be an edge, and let $b_1 b_2 a_1 a_2$ and $c_1 c_2 a_1 a_2$ be two copies of $K_4$ whose edge sets intersect solely in $a_1 a_2$; in particular, the vertices $b_1, b_2, c_1, c_2$ are all distinct.
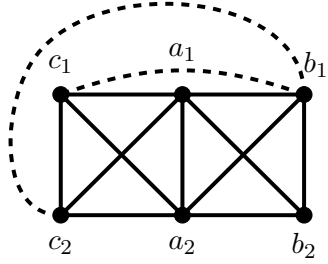
\begin{figure}[htb]
\centering

\begin{tikzpicture}[scale=0.8]
\foreach \a in {0,1,2}{
	\draw[fill] (-2 + 2*\a , 1) circle (0.13);
	\draw[fill] (-2 + 2*\a , -1) circle (0.13);
}

\draw[line width=1.5pt] (-2,1) -- (0,1) -- (2,1) -- (2,-1)
						-- (0,-1) -- (-2,-1) -- (-2,1)
						-- (0,-1) -- (2,1);
\draw[line width=1.5pt] (-2,-1) -- (0,1) -- (2,-1);
\draw[line width=1.5pt] (0,1) -- (0,-1);

\draw[line width=1.5pt, dashed] (2,1) to[out=160, in=20] (-2,1);
\draw[line width=1.5pt, dashed] (2,1) 
						to[out=90, in=60] (-2.5,1.5)
						to[out=-120,in=180] (-2,-1);

\draw (0,0) node at (0,1.7) {$a_1$}
			node at (0,-1.5) {$a_2$}
			node at (2.2, 1.5) {$b_1$}
			node at (2, -1.5) {$b_2$}
			node at (-2, 1.5) {$c_1$}
			node at (-2,-1.5) {$c_2$};
\end{tikzpicture}
\end{figure}  
  
Now, the edge $a_1 b_1$ should also participate in two copies of $K_4$ where it is the sole intersection, $a_1 b_1 v v'$ and $a_1 b_1 u u'$. However the vertex $a_1$ is of degree $5$ so $\{v,v',u,u'\} = \{a_2, b_2, c_1, c_2 \}$ meaning that $b_1$ is adjacent to $c_1$ and $c_2$. Repeating the same argument for $a_1 b_2$, we get that $b_2$ is adjacent to $c_1$ and $c_2$ meaning that $\{a_1, a_2, b_1, b_2, c_1, c_2\}$ induces a $K_6$. Since $G$ is connected and $5$-regular, we have $G = K_6$.
\end{proof}

\begin{claim}
  \label{claim:C4-core-is-K33-or-Q3}
  If $G$ is a connected $C_4$-core with $e_G/v_G \le m_2(C_4)$, then $G$ is $3$-regular.
\end{claim}

\begin{proof}
  Suppose that $G$ is a connected $C_4$-core with $e_G/v_G \le m_2(C_4) = 3/2$.  Since $\delta(G) \ge 2\delta(C_4)-1 = 3 \ge 2e_G/v_G$, the graph $G$ must be $3$-regular.
\end{proof}

The proof of Proposition~\ref{prop:cores-single-graph} is now
complete.  Since it is easily checked that $K_6$ is not Ramsey for
$K_4$ and $K_5$ is not Ramsey for $K_3$, the following two claims complete the proof of the $0$-statement in Theorem~\ref{thm:rodl-rucinski}.

\begin{claim}
  \label{claim:K5}
  If $G \neq K_5$ is minimally (list-)Ramsey graph for $K_3$, then $e_G/v_G > m_2(K_3)$.
\end{claim}

\begin{proof}
  Suppose that a graph $G$ is minimally (list-)Ramsey for $K_3$ and satisfies $e_G/v_G \le m_2(K_3) = 2$.  If, for some $v \in V(G)$, there was an orientation of the edges of $G[N(v)]$ with maximum out-degree at most one, we could then extend every $K_3$-free colouring of $G - v$ to $G$ as follows: For every $u \in N(v)$, the edge $uv$ gets a colour that is different from the colour of the out-edge from $u$.  (Since each triangle involving $v$ contains an edge of $G[N(v)]$, this colouring is $K_3$-free.)  As every graph with at most four edges has such an orientation, we may assume that $e(N(v)) \ge 5$ for every $v \in V(G)$; in particular, $\delta(G) \ge 4 \ge 2e_G/v_G$, so $G$ must be $4$-regular.

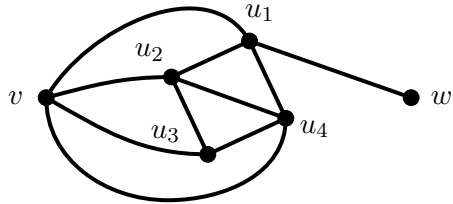
\begin{figure}[htb]
\centering

\begin{tikzpicture}[scale=0.8]
\foreach \a in {0, 1, ..., 3}{
	\draw[fill] (-20 + 90* \a:1) circle (0.13);
}

\draw[fill] (-3, 0) circle (0.13);
\draw[fill] (3, 0) circle (0.13);

\draw[line width=1.5pt] (-3,0) to[out=60, in=120] (70:1);
\draw[line width=1.5pt] (-3,0) to[out=15, in=180] (90+70:1);
\draw[line width=1.5pt] (-3,0) to[out=-30, in=180] (2*90+70:1);
\draw[line width=1.5pt] (-3,0) 
						to[out=-90, in=180] (-1,-1.7)
						to[out=0, in=-90] (3*90+70:1);

\draw[line width=1.5pt] (180-20:1) -- (270-20:1) -- (-20:1) -- (90-20:1) -- (180-20:1) -- (-20:1);
\draw[line width=1.5pt] (90-20:1) -- (3,0);

\draw (0,0) node at (70:1.5) {$u_1$}
			node at (90+60:1.5) {$u_2$}
			node at (270-60:1.2) {$u_3$}
			node at (-20:1.5) {$u_4$}
			node at (-3.5,0) {$v$}
			node at (3.5,0) {$w$};
\end{tikzpicture}
\end{figure}  

  If $e(N(v)) > 5$ for some $v \in V(G)$, then $G = K_5$, since $G$ is $4$-regular and connected.  We may thus further assume that $e(N(v)) = 5$ for every $v$.  Pick some $v$ and denote $N(v) = \{u_1, u_2, u_3, u_4\}$ so that $u_1u_3 \notin G[N(v)]$. Since $G$ is $4$-regular, there must be a $w \in V(G) \setminus (\{v\} \cup N(v))$ such that $N(u_1) = \{v, w, u_2, u_4\}$. Moreover, $w$ is not adjacent to either of $v$,  $u_2$, and $u_4$, as they all have $4$ neighbours in $\{v\} \cup N(v)$, and thus $e(N(u_1)) \le 3$, a contradiction.
\end{proof}

\begin{claim}
  \label{claim:3-regular-2-colouring}
  Every graph with maximum degree three can be $2$-list-coloured in such a~way that
  every colour class is a forest.
\end{claim}

\begin{proof}
  Suppose that this were not true and let $G$ be a smallest graph with maximum degree three that does not admit such a cycle-free colouring for some choice of lists of size two assigned to its edges.  If some vertex $v$ had degree at most two or its three incident edges do not all have the same lists, then we could extend any colouring of $G-v$ by assigning different colours to the edges incident with $v$; this way we would not create any monochromatic cycles.  Therefore, $G$ must be $3$-regular and, as $G$ is clearly connected, all lists must be identical, say $\{\text{red}, \text{blue}\}$.  Let $v_1 v_2 \dotsc v_\ell v_1$ be an arbitrary cycle in $G$ and denote by $e_1, \dotsc, e_\ell$ the edges incident with $v_1, \dotsc, v_\ell$, respectively, that are not on this cycle.  We can extend any cycle-free colouring of $G-\{v_1, \dotsc, v_\ell\}$ to a cycle-free colouring of $G$ by colouring all the edges of the path $v_1 \dotsc v_\ell$ and the edge $e_1$ red and the edges $v_1 v_\ell$ and $e_2, \dotsc, e_\ell$ blue.
\end{proof}

\label{page:end-of-proof-of-RR}
We have now completed the proof of the $0$-statement in Theorem~\ref{thm:rodl-rucinski}.  The following two simple claims are what remains to be proved to obtain the complete statement of the Deterministic Lemma and, as a result, the proof of the $0$-statement in the list-colouring version of Theorem~\ref{thm:rodl-rucinski}.

\begin{claim}
  $K_6$ is not list-Ramsey with respect to $K_4$.
\end{claim}

\begin{proof}
Given an assignment of lists of size two, pick a random colouring uniformly from the lists. There are $\binom{6}{4} = 15$ copies of $K_4$ in $K_6$, so the expected number of monochromatic $K_4$s is at most $15 \cdot 2^{1-6} < 1$, meaning that there is a proper list-colouring.
\end{proof}

\begin{claim}
  \label{claim:K5-list-colourable}
  $K_5$ is not list-Ramsey with respect to $K_3$.
\end{claim}

\begin{proof}
If some colour, say red, contains a $5$-cycle, then we may colour this $5$-cycle red and the complementary $5$-cycle not red.  If some colour class, say red, contains an edge, say $e$, not in a triangle, then we may colour $K_5 \setminus e$ without monochromatic triangles (this is possible as no proper subgraph $K_5$ is list-Ramsey for $K_3$, by Claim~\ref{claim:K5}) and colour $e$ red.  If none of the above is true, then each colour induces one of the following graphs: $K_3$, $K_4$, $K_4^-$, $K_5 \setminus K_3$, or two triangles sharing a vertex.  If some colour, say red, induces $K_5 \setminus K_3$, then we colour $K_{2,3}$ with red, the remaining edge of $K_5 \setminus K_3$ with not red and the edges of the $K_3$ in the complement with two different colours other than red.  If one of the colours, say red, induces $K_4$ or $K_4^-$, then colour a $C_4$ with red and its diagonal(s) with a colour other than red.  Each of the remaining, uncoloured four edges can close at most one monochromatic triangle,  as red is not available anywhere outside of the $K_4$ we have already coloured; thus we may colour them one-by-one.  This leaves the case where every colour class is either $K_3$ or two triangles sharing a vertex.  But this is impossible, since $3$ does not divide $2e(G) = 20$.
\end{proof}

\subsection{General families}
\label{sec:general-families}

We start by proving the Deterministic Lemma for families $\cF$ with $m_2(\cF) > 2$.  To this end, we first prove Proposition~\ref{prop:cores-family} and then show that $K_6$ is not list-Ramsey with respect to any family $\cF$ with $m_2(\cF) \ge 5/2 = e_{K_6}/v_{K_6}$.

\begin{proof}[Proof of Proposition~\ref{prop:cores-family}]
  Write $m_2(\cF)= k+ \eps$, where $k \ge 2$ is an integer and $\eps
  \in [0,1)$, and suppose that $G$ is a connected $\cF$-core with
  $e_G/v_G \le k+\eps$.  It suffices to show that $K_4 \in \cF$ and $G$ is a $K_4$-core,
  since then Claim~\ref{claim:K4-core-is-K6} will allow us to conclude
  that $G = K_6$.
  We split the argument into three cases, depending on which item of
  Lemma~\ref{lem:discharging} holds.  Since $K_3$, $C_4$, and $K_4$
  are the only strictly $2$-balanced graphs with at most four
  vertices, the assumption that $m_2(\cF) > 2$ implies that the only
  graph with fewer than five vertices that can belong to $\cF$ is $K_4$.

  \medskip
  \noindent
  \textit{Case 1.}
  Suppose that $G$ has a vertex $v$ of degree at most $2k$ and let $e$ be an edge incident to $v$.  Since $G$ is an $\cF$-core, there are two copies of some $F, F' \in \cF$ in $G$ that intersect solely in $e$, and thus $\deg(v) \ge \delta(F)+\delta(F')-1$.  Since $v_F, v_{F'} > 3$, we can apply the Helpful Lemma to learn that $\delta(F), \delta(F') \ge k+1$, which gives $\deg(v) \ge 2k+1$, a contradiction.

  \medskip
  \noindent
  \textit{Case 2.}
  Suppose now that $G$ has an edge $uv$ with $\deg(u) = 2k+1$ and
  $\deg(v) \le 2k+2$ and that $\eps \ge 1/2$.  By assumption, there are two copies of some $F,
  F' \in \cF$ that intersect only in $uv$.  If $v_F, v_{F'} > 4$, we
  could use the Helpful Lemma to learn that the endpoints of every
  edge of $F$ and $F'$ touch strictly more than $2m_2(\cF) \ge
  2(k+\eps) \ge 2k+1$ edges.  This would mean that $\{u, v\}$ touch at
  least $2(2k+2)-1 = 4k+3 > \deg(u) + \deg(v) - 1$ edges in $G$, a
  contradiction.  Thus, one of $F, F'$ is a $K_4$.  Consequently,
  we have
  \[
    e_G/v_G \le m_2(\cF) \le m_2(K_4) = 5/2.
  \]
  As Case 1 rules out $\delta(G) \le 4$, the graph $G$ must be
  $5$-regular and $m_2(\cF) = 5/2$.

  Finally, we show that $G$ is a $K_4$-core.  If this were not true, then some edge $uv$ of $G$ would be the sole intersection of two copies of some graphs from $\cF$, not both of them $K_4$.  The Helpful Lemma implies that, for every $F \in \cF$ with $v_{F} > 4$, the endpoints of every edge of $F$ touch strictly more than $5$ edges.  On the other hand, the endpoints of every edge of $K_4$ touch precisely $5$ edges.  This means that at least $10$ edges of $G$ would have to touch $\{u, v\}$, a contradiction.

  \medskip
  \noindent
  \textit{Case 3.}
  Finally, suppose that $\eps \ge 7/8$ and that there are vertices $u, v_1, v_2$ with $\deg(u)
  = 2k+3$, $\deg(v_1) = \deg(v_2) = 2k+1$, and $uv_1, uv_2 \in G$.  We
  may assume that $v_1v_2 \notin G$, since otherwise we would be in
  Case~2.  By assumption, there are two copies of some $F_1, F_2 \in
  \cF$ that intersect only in $uv_1$.  Applying the Helpful Lemma, we
  learn that $\delta(F_1) > m_2(F_1)$ and $\delta(F_2) > m_2(F_2)$.
  If $m_2(F_i) \ge k+1$ for some $i \in \{1, 2\}$, then we would
  have $\deg(v_1) \ge \delta(F_1) + \delta(F_2) - 1 \ge 2k+2$, a
  contradiction.  Therefore,
  \[
    k+ 7/8 \le k + \eps =m_2(\cF) \le m_2(F_i) < k+1
  \]
  for each $i \in \{1, 2\}$, and thus the definition of $m_2(\cdot)$
  implies that $v_{F_1}, v_{F_2} \ge 10$.

  We can therefore apply the Helpful Lemma and learn that the endpoints of every edge of $F_1$ and $F_2$ touch strictly more than $2m_2(\cF) \ge 2(k+\eps) \ge 2k+1$ edges.  Because our copies of $F_1$ and $F_2$ intersect only in $uv$ and $\{u,v\}$ touches exactly $4k+3$ edges of $G$, one of the copies, say of $F_2$, contains $uv_2$.  By the Helpful Lemma, this copy of $F_2$ must therefore use strictly more than $3m_2(F_2) \ge 3(k+\eps) \ge 3k+2$ edges touching $\{u, v_1, v_2\}$.

  Finally, since $G$ is an $\cF$-core, there is another copy of some
  graph $F_3 \in \cF$ that intersects our copy of $F_2$ solely in
  $uv_2$.  The Helpful Lemma gives $\delta(F_3) \ge \lceil m_2(F_3)
  \rceil \ge k+1$.  The union of our copies of $F_1$, $F_2$, and $F_3$
  must contain at least $3k+3+2k+2+k+1 - 2 = 6k+4$ edges touching
  $\{u, v_1, v_2\}$.  However, the bounds on the degrees of $u$,
  $v_1$, and $v_2$ in $G$ imply that they touch only $6k+3$ edges, a
  contradiction.
\end{proof}

\begin{claim}
  \label{claim:K6-not-list-Ramsey}
  $K_6$ is not list-Ramsey with respect to any family $\cF$ with $m_2(\cF) \ge 5/2$.
\end{claim}

\begin{proof}
  If all lists assigned to the edges of $K_6$ are identical, then we can colour so that one colour class is $K_{3,3}$ and the other $2 \cdot K_3$; both these graphs have $2$-density two.  We may thus assume that the lists incident to some vertex $v$ are not all identical.  Any $\cF$-free colouring of $K_6-v$ (at least one such colouring exists by Proposition~\ref{prop:cores-family}) can now be extended to $K_6$ by colouring the edges incident with $v$ so that no colour is repeated more than twice.  Since every graph $F$ with $m_2(F) \ge 5/2$ has minimum degree at least three, this colouring is also $\cF$-free.
\end{proof}

Finally, we prove Proposition~\ref{prop:eGvG-more-than-two}, thus establishing the Deterministic Lemma for families $\cF$ with $m_2(\cF) \le 2$.

\begin{proof}[{Proof of Proposition~\ref{prop:eGvG-more-than-two}}]
  Suppose by contradiction that a graph $G$ is minimally list-Ramsey for some family $\cF$ of graphs, each containing a cycle, and satisfies $e_G/v_G \le \min\{m_2(\cF), 2\}$.  Observe first that $\delta(G) \ge 3$ and that, for every vertex $v$ of degree three, the lists assigned to the edges incident with $v$ are identical.  Indeed, if this were not true, we could extend any $\cF$-free colouring of $G-v$ to $G$ by assigning different colours to all edges incident with $v$; such a colouring would also be $\cF$-free as every graph in $\cF$ has minimum degree at least two.  We may also assume that $m_2(\cF) \ge e_G/v_G > 3/2$, as otherwise $G$ is $3$-regular and Claim~\ref{claim:3-regular-2-colouring} implies that it is not list-Ramsey for $\cF$.  In particular, since $m_2(C_4) = 3/2$, the only graphs with fewer than five vertices that can belong to $\cF$ are $K_3$ and $K_4$.

\medskip
\noindent
\textit{Case 1. $e_G/v_G \in (3/2, 2)$.}
We apply Lemma~\ref{lem:discharging} to $G$ with $k+\eps = e_G/v_G$.  Since we have already ruled out the possibility that $\delta(G) \le 2$, there are only two subcases.

\medskip
\noindent
\textit{Case 1a.}
Suppose first that $G$ has an edge $uv$ with $\deg(u) = 3$ and $\deg(v) \le 4$.  Recall that the three lists assigned to the edges incident with $u$ are the same, say $\{\text{red}, \text{blue}\}$.  Consider an arbitrary $\cF$-free colouring of $G-u$.  Since $v$ has at most three neighbours in $G-u$, it is incident with at most one blue edge or at most one red edge; without loss of generality, this rare colour is blue.  Let $w$ be the other endpoint of the blue edge incident with $v$.  Colour two edges incident with $u$ blue and one edge red so that $uw$, if it is an edge of $G$ at all, is not blue.

If this colouring was Ramsey for $\cF$, we would see a monochromatic copy $F$ of a graph from $\cF$ that contains $u$.  As $F$ has minimum degree two, it cannot be red and, if it is blue, it must also contain both $v$ and $w$ and the other blue neighbour of $u$ (which is distinct from $w$).  However,  since $m_2(F) > 3/2$, it cannot contain a path of length three whose two centre vertices have degree two; indeed $\delta(K_4) > 2$ and the Helpful Lemma implies that, for every $F \in \cF$ with $v_F > 4$ and every pair $\{a,b\}$ of vertices of $F$, we have $\bar{e}_F(\{a,b\}) > 3$.

\medskip
\noindent
\textit{Case 1b.}
Suppose now that $\eps \ge 7/8$ and that $G$ has two edges $vu_1$ and $vu_2$, where $v$ is of degree $5$ and $u_1, u_2$ are non-adjacent vertices of degree $3$.  Recall that, for both $i \in \{1, 2\}$, the three lists assigned to the edges incident with $u_i$ are identical. Consider an arbitrary $\cF$-free colouring of $G - \{u_1, u_2\}$.  Since $v$ has degree three in this graph, for each $i \in \{1, 2\}$, there is a colour $c_i$ belonging to the lists seen at $u_i$ that appears on at most one edge of $G - \{u_1, u_2\}$ incident with $v$ (it may happen that $c_1 = c_2$).  Now, for each $i \in \{1, 2\}$, colour two edges incident to $u_i$ with $c_i$ (and the third edge in the other available colour) so that there is no triangle whose all edges are coloured $c_i$; this is possible as $v$ had at most one incident edge coloured $c_i$ and $u_1$ and $u_2$ are not adjacent.

If this colouring was Ramsey for $\cF$, we would see a monochromatic copy $F$ of a graph from $\cF$ that contains $u_i$ for some $i\in\{1,2\}$.  As $F$ has minimum degree two, it must be coloured $c_i$ and it must contain both $v$ and at least one more neighbour of $v$ connected by an edge coloured $c_i$.  If there is only one such neighbour, then $F$ would contain a path of length three whose two centre vertices have degree two, which is impossible (see Case 1a).  Otherwise, it must be that $c_1 = c_2$, there are exactly two such neighbours, and $\bar{e}_F(\{v, u_1, u_2\}) = 5$.  However, since $m_2(F) \ge m_2(\cF) \ge 1+\eps \ge 2 - 1/8$ but $F \neq K_3$, then either $m_2(F) \ge 2$, in which case $\delta(F) \ge 3$, or $15/8 \le m_2(F) < 2$, in which case $F$ has at least ten vertices.  In both cases, the Helpful Lemma implies that $\bar{e}_F(\{v, u_1, u_2\}) \ge 6$, a contradiction.

\medskip
\noindent
\textit{Case 2. $e_G/v_G = 2$.}
Since we have already ruled out the possibility that $\delta(G) \le 2$, either $\delta(G) = 3$ or $G$ is $4$-regular.  Let $u$ be a vertex of smallest degree in $G$ that minimises $e(N(u))$ and consider an arbitrary $\cF$-free colouring of $G-u$.  We will show that it can be extended to an $\cF$-free colouring of $G$, unless $G = K_5$.  As every $F \in \cF$, other than the triangle, has minimum degree at least three, it is enough to show there is a colouring of the edges incident with $u$ that uses each colour at most twice and has no monochromatic triangles containing $u$.  It is straightforward to check that such a colouring exists when $\deg(u) = 3$, so we will assume that $G$ is $4$-regular.

\medskip
\noindent
\textit{Case 2a. $e(N(u)) \le 4$.}
Denote the four neighbours of $u$ by $v_1, \dotsc, v_4$.  If the four lists assigned to $uv_1, \dotsc, uv_4$ are identical, say $\{\text{red}, \text{blue}\}$, then we can colour as follows:  Since $N(u)$ is not a clique, we may assume that $v_1v_2 \notin G$.  If $v_3v_4 \in G$ and it is coloured red, colour $uv_3, uv_4$ blue and $uv_1, uv_2$ red; otherwise, colour $uv_3, uv_4$ red and $uv_1, uv_2$ blue.  It is straightforward to check that no monochromatic triangle was created.

If the four lists assigned to $uv_1, \dotsc, uv_4$ are not identical, then some colour, say red, appears in the lists of at most two edges incident with $u$.  If red appears only once, say in the list of $uv_1$, then we colour $uv_1$ red and the remaining three edges as in the case where $\deg(u) = 3$.  Suppose now that red appears in two lists, say of $uv_1$ and $uv_2$. Since $N(u)$ induces at most four edges, either $v_3v_4 \notin G$ or one of $v_1, v_2$ is not adjacent to one of $v_3, v_4$.  If $v_3v_4 \notin G$, then we colour $uv_1$ red, $uv_2$ with a colour $c$ different than red and both $v_3$ and $v_4$ with a colour different than $c$;  it is easy to check that no colour is used more than twice and that no triangle is monochromatic.  Otherwise, we may assume without loss of generality that $v_2 v_4 \notin G$.  Now, colour $uv_1$ red, $uv_2$ with a colour $c$ different than red, $uv_3$ with a colour $c' \neq c$, and $uv_4$ with a colour different than $c'$.  Again, it is easy to check that no colour was used more than twice and that no triangle is monochromatic.

\medskip
\noindent
\textit{Case 2b. $e(N(u)) \ge 5$.}
Since $G$ is connected, $4$-regular, and $e(N(v)) \ge 5$ for all $v \in V(G)$, it must be $K_5$, see Claim~\ref{claim:K5}.  Thus, it is enough to show that $K_5$ is not list-Ramsey for any family $\cF$ with $m_2(\cF) \ge 2$.  To this end, we first claim that every $2$-balanced graph with $2$-density at least $2$ that is contained in $K_5$ contains a triangle.  Indeed, if $F$ is strictly $2$-balanced and $m_2(F) \ge 2$, then $v_F \ge 3$ and $e_F \ge 2(v_F - 2) + 1 = 2v_F-3$, but every triangle-free graph $F$ satisfies $e_F \le v_F^2/4$;  it is easy to check that $2v_F - 3 > v_F^2/4$ when $v_F \in \{3, 4, 5\}$.  Consequently, any triangle-free colouring of $K_5$ will automatically be $\cF$-free for every family $\cF$ with $m_2(\cF) \ge 2$.  Finally, the fact that $K_5$ is $2$-list-colourable with respect to $K_3$ was proved in Claim~\ref{claim:K5-list-colourable}.
\end{proof}

\section{Allowing forests}
\label{sec:forests}

This section is dedicated to the case where our family contains a forest. We first take care of the special case of star forests by proving Theorem~\ref{thm:star_forests}. The rest of the section is dedicated to Theorems~\ref{thm:general-forests} and~\ref{thm:non-list-forests}.

\subsection{Star forests}

\begin{proof}[Proof of Theorem~\ref{thm:star_forests}]

For the $1$-statement, it is enough to prove the result for the usual Ramsey property. Recall that $\cF$ contains a star forest $F$ where the number of rays in each star is at most $D$. Let $k$ count the number of stars in $F$. When $p \gg n^{-(1 + 1/s)}$, the random graph will a.a.s.\ contain $r \cdot (k-1) + 1$ disjoint stars with $s$ rays (or indeed, any constant number of copies of this star). By the pigeonhole principle, any $r$-colouring of the edges of any of these stars will give a monochromatic $D$-star. By the pigeonhole principle again, at least $k$ of them must use the same colour, giving a monochromatic copy of $F$.

For the $0$-statement, assume $p \ll n^{-(1 + 1/s)}$. At this density, the connected components of $G_{n,p}$ are a.a.s.\ all trees with at most $s-1$ edges. We claim that one can colour the edges of any graph with this property from any assignment of lists of $r$ colours such that the monochromatic components are all star forests, and the maximum degree in each colour is strictly less than $D$. This would imply that the colouring does not contain any monochromatic copy of a member of $\cF$.

For every component, we start by taking a vertex $u$ of maximum degree. Colour the edges touching this vertex while making sure that we do not use the same colour $D$ times. Since its degree is at most $s-1 = r\cdot (D-1)$, we can always do this. Next, if we arrive at a new vertex $v$, we will colour the new edges that touch $v$ while making sure never to use the colour of the edge that introduced this vertex. This would ensure that all monochromatic components are star forests. In the worst case, this restriction gives us $r-1$ available colours for each new edge. 

The same argument would mean that, as long as the number of new edges introduced at $v$ is at most $(r-1)\cdot(D-1)$, we can still avoid using the same colour $D$ times. Luckily, this is always the case. Indeed, if $v$ introduced at least $(r-1)(D-1) + 1$ new edges, then $\deg(v) \ge (r-1)(D-1) + 2$ (counting the edge that led us to this vertex). Recall that by maximality $\deg(u) \ge \deg(v)$. This means that our connected component contains at least $2(r-1)(D-1) + 4 - 1$ edges, and that is strictly greater than $r(D-1) = s-1$ for all $r \ge 2$, giving a contradiction. 
\end{proof}

\subsection{Nontrivial forests}

Even after excluding star forests, if $m_2(\cF) = 1$ we can no longer make the reduction to strictly $2$-balanced graphs. Instead, we will replace the Probabilistic Lemma with the following well-known fact about sparse random graphs (see~\cite[Section~5]{JanLucRuc00}), which can be readily proved using an argument similar to the one we used to prove the Probabilistic Lemma.

\begin{lemma}
  There exists a constant $c > 0$ such that, whenever $p \le c n^{-1}$, a.a.s.\  every connected component of $G_{n,p}$ has at most one cycle.
\end{lemma}

In view of the lemma, in order to prove the bulk of Theorems~\ref{thm:general-forests} and~\ref{thm:non-list-forests}, it suffices to devise a colouring scheme for graphs whose connected components are either trees or unicyclic graphs, i.e., graphs that consist of exactly one cycle, with trees rooted at some of its vertices.  Note that, for any graph $G$, the condition $m(G) \le 1$ is synonymous with $G$ having the aforementioned property.

\begin{prop}
  \label{prop:BC}
  Let $\cF$ be a family of graphs and let $r \ge 2$ be an integer. Then every graph $G$ with $m(G) \le 1$ is not $r$-list-Ramsey with respect to $\cF$ unless $r = 2$ and $\cF$ contains both a $\cB$-graph and a $\cC^*$-graph.
\end{prop}

Further, since every graph $G$ with $m(G)=1$ appears in $G_{n,p}$ with probability $\Omega(1)$ whenever $p = \Omega(n^{-1})$, the following proposition completes the proof of Theorem~\ref{thm:general-forests}.

\begin{prop}
  \label{prop:BC-inverse}
  For every family $\cF$ of graphs that contains both a $\cB$-graph and a $\cC^*$-graph, there exists a graph $G$ with $m(G) \le 1$ that is list-Ramsey with respect to~$\cF$.
\end{prop}

Finally, we turn to the usual (non-list) Ramsey threshold for families of graphs that contain both a $\cB$-graph and a $\cC^*$-graph.  The classification there is slightly more complex, and is described using an auxiliary hypergraph.

\begin{dfn}
  Given a family of graphs $\cF$ that contains both a $\cB$-graph and a $\cC^*$-graph, we define an auxiliary hypergraph $\cA = \Aux(\cF)$ in the following way.  The vertices of $\cA$ are all odd cycles $C_{2\ell+1}$, where $\ell \ge 1$ is an integer.  Further, for every $F \in \cF$ that is a $\cC^*$-graph whose non-star components are $F^1, \dotsc, F^m$, we add to $\cA$ the edge $\{C^1, \dotsc, C^m\}$ for every choice of (not necessarily distinct) odd cycles $C^1, \dotsc, C^m$ such that $F^i \subseteq C^i$ for all $i$.
\end{dfn}

For example, let $C^1$ and $C^2$ be two odd cycles of different lengths and let $B$ be a broom with at least two hairs. Then, taking a family $\cF_1$ which consists of the broom $B$ and the disjoint union $C^1 \cup C^2$ will give us $\Aux(\cF_1)$ which is isomorphic to the graph $K_2$. If we take another odd cycle $C^3$ such that the lengths are still distinct, then $\cF_2 \coloneqq \{ B, C^1 \cup C^2, C^1 \cup C^3, C^2 \cup C^3\}$ will have $\Aux(\cF_2) \cong K_3$. One last example, showcasing the definition for paths, is the family $\cF_3 \coloneqq \{B, C^1 \cup P \}$, where $P$ is a path with at least four vertices.  In this case, $\Aux(\cF_3)$ is infinite, with an edge $\{C^1, C\}$ for every odd cycle $C$ with $v(C) \ge v(P)$.

The following proposition completes the proof of Theorem~\ref{thm:non-list-forests}.

\begin{prop}
  \label{prop:BC-non-list}
  Let $\cF$ be a family of graphs that contains both a $\cB$-graph and a $\cC^*$-graph.  There exists a graph $G$ with $m(G) \le 1$ that is $2$-Ramsey for $\cF$ if and only if $\cF$ contains a star forest or $\Aux(\cF)$ is not $2$-colourable.
\end{prop}

Carrying on with our example from before, $\Aux(\cF_1)$ is $2$-colourable, while $\Aux(\cF_2)$ is not, meaning that the Deterministic Lemma would hold for the former and fail for latter. Note that $\Aux(\cF_3)$ is $2$-colourable iff $v(P) > v(C^1)$. Indeed,  when $v(P) \le v(C^1)$, then $\Aux(\cF_3)$ contains the singleton edge $\{C^1\}$.

\begin{remark}
Note that the combination of these propositions implies a discrepancy in the behaviour of the thresholds for the list-Ramsey and the usual Ramsey properties.  Indeed, for $\cF_1$ both thresholds appear at $n^{-1}$, but while the usual Ramsey threshold is semi-sharp, $G_{n,c/n}$ is $2$-list-Ramsey for $\cF_1$ with probability $\Omega(1)$ for every $c > 0$.
\end{remark}

We prove Propositions~\ref{prop:BC}--\ref{prop:BC-non-list} in the three subsections that follow.

\subsection{Proof of Proposition~\ref{prop:BC}}
\label{sec:proof-proposition-BC}

A connected graph $G$ with $m(G) \le 1$ whose every edge is assigned a list of $r \ge 2$ colours is called \emph{nontrivial} if it contains an odd cycle, $r=2$, and all the lists of colours for edges in the cycle are identical; otherwise, we call $G$ \emph{trivial}.

\begin{claim}
  \label{claim:trivial-colouring}
  Suppose that every edge of a connected graph $G$ with $m(G) \le 1$ is assigned a list of $r \ge 2$ colours.  There is a colouring from these lists such that every monochromatic subgraph is a star forest unless $G$ is nontrivial.
\end{claim}

\begin{proof}
  Fix an arbitrary orientation of $G$ with out-degree at most one. If an edge $uv$, oriented from $u$ to $v$, is coloured differently than the the colour of the out-edge from $v$ (or if $v$ has out-degree zero), we say that $uv$ is \emph{nonrepetitive}.  Our aim is to find a nonrepetitive coloring for all of $G$, i.e., choose colours so that all edges are nonrepetitive.
   Since each connected graph that is not a star contains either a path of length three or a triangle and since each orientation of the path of length three and of $K_3$ with out-degree at most one contains a directed path of length two, the colour classes of a nonrepetitive colouring are guaranteed to be star forests.

Given any graph $H \subseteq G$ whose edges are coloured arbitrarily, and any edge $uv$ oriented from $u \notin V(H)$ to $v \in V(H)$, one can find a nonrepetitive colour for $uv$.  Since our $G$ can be built from a (directed) cycle or a single vertex using a sequence of such operations, the only reason $G$ might not admit a nonrepetitive colouring is if it contains a cycle which does not admit a nonrepetitive colouring.

  Suppose now that $G$ is a (directed) cycle.  If either $r \ge 3$ or not all colour lists are identical, then we can construct a nonrepetitive colouring as follows.   Denote the vertices of $G$ by $v_1, \dotsc, v_\ell$ so that its directed edges are $(v_i,v_{i+1})$, for $i \in \{1, \dotsc, \ell\}$, where the addition is modulo $\ell$.  Our assumption implies that there is an $i$ and a colour $c$ that belongs to the list of $v_iv_{i+1}$ such that the list of $v_{i+1}v_{i+2}$ contains at least two colours different than $c$; without loss of generality, we may assume that $i = 1$.  We now colour $v_1v_2$ with $c$, then the edges $v_\ell v_1, v_{\ell-1} v_\ell, \dotsc, v_3v_4$, and finally the edge $v_2v_3$ with a colour different than $c$ and the colour chosen for $v_3v_4$.  Finally, if $r = 2$ and all the lists are identical, then a nonrepetitive colouring exists as long as $G$ is an even cycle.
\end{proof}

\begin{proof}[Proof of Proposition~\ref{prop:BC}]

By Claim~\ref{claim:trivial-colouring}, we can colour every trivial component of $G$ such that every resulting monochromatic component will be a star forest. Therefore we will only worry about the nontrivial components of $G$. Any nontrivial component contains an odd cycle and the edges of the cycle use identical lists of two colours, say red and blue.

If $\cF$ does not contain a $\cC^*$-graph, colour the entire cycle red. After we colour the cycle, we can extend the colouring to the other edges of the component. Label the vertices of the cycle as $v_1, \dots, v_{2k+1}$ and orient its edges from $v_i$ to $v_{i+1}$ (where the addition is modulo $2k+1$).  We extend the colouring to all of $G$ by adding the edges not on the cycle one by one as leaves and choosing a nonrepetitive colour for each one, exactly as in the previous proof. This way the monochromatic components are either odd cycles or stars, meaning that we are $\cF$-free.

If $\cF$ does not contain a $\cB$-graph, colour the edges $v_1 v_2, v_3 v_4, \dotsc, v_{2k - 1} v_{2k}$ red and the remaining edges of the cycle blue (so that the only monochromatic path of length two is the blue path $v_{2k} v_{2k +1} v_1$). We use the same method as above to extend the colouring to the other edges of the component. Note that in any colour other than blue, any resulting monochromatic component will be a star. In the blue component however, the edges of the tree rooted at $v_1$ may also be coloured blue (since we only forbid the colour red, used for $v_1 v_2$). Therefore we may have monochromatic brooms. However, since $\cF$ does not contain $\cB$-graphs we will be $\cF$-free.
\end{proof}

\definecolor{reddish}{HTML}{FB5039}
\definecolor{bluish}{HTML}{1F44BC}

\begin{figure}[htb]
  \centering

  \begin{tikzpicture}[scale=0.8]

    \foreach \a in {1, 2, 3, 4, 5}{
      \draw[line width=1.5pt, color=reddish] 
      (72*\a - 72 -20:1.2) -- (72*\a -20:1.2);

      \foreach \b in {-15,0,15}{
      
        \draw[line width=1.5pt, color=bluish] 
        (\b -20 + 72* \a:2.3) -- (72*\a -20:1.2);
        
        \draw[line width=1.5pt, color=reddish]
        (\b -20 + 72* \a:2.3) -- (5 + \b -20 + 72* \a:3)
        (\b -20 + 72* \a:2.3) -- (-5 + \b -20 + 72* \a:3);
      }
    }

    \foreach \a in {0, 1, ..., 4}{
      \draw[fill] (-20 + 72* \a:1.2) circle (0.13)
      (-20 + 72* \a:2.3) circle (0.13)
      (15-20 + 72* \a:2.2) circle (0.13)
      (-15-20 + 72* \a:2.2) circle (0.13);
    }

    \draw (0,0) node at (-72:1.4) {$v_1$}
    node at (-72*2:1.4) {$v_2$}
    node at (-72*3:1.4) {$v_3$}
    node at (-72*4:1.4) {$v_4$}
    node at (-72*5:1.4) {$v_5$};

  \end{tikzpicture}
  \begin{tikzpicture}[scale=0.3]
    \draw[fill, transparent] (0,0) circle (0.2);
    \draw[fill, transparent] (6,0) circle (0.2);
  \end{tikzpicture}
  \begin{tikzpicture}[scale=0.8]

    \foreach \a in {1,3,5}{
      \draw[line width=1.5pt, color=bluish] 
      (72*\a - 72 -20:1.2) -- (72*\a -20:1.2);

      \foreach \b in {-15,0,15}{
      
        \draw[line width=1.5pt, color=reddish] 
        (\b -20 + 72* \a:2.3) -- (72*\a -20:1.2);
        
        \draw[line width=1.5pt, color=bluish]
        (\b -20 + 72* \a:2.3) -- (5 + \b -20 + 72* \a:3)
        (\b -20 + 72* \a:2.3) -- (-5 + \b -20 + 72* \a:3);
      }
    }

    \foreach \a in {2,4}{
      \draw[line width=1.5pt, color=reddish] 
      (72*\a - 72 -20:1.2) -- (72*\a -20:1.2);

      \foreach \b in {-15,0,15}{
      
        \draw[line width=1.5pt, color=bluish] 
        (\b -20 + 72* \a:2.3) -- (72*\a -20:1.2);
        
        \draw[line width=1.5pt, color=reddish]
        (\b -20 + 72* \a:2.3) -- (5 + \b -20 + 72* \a:3)
        (\b -20 + 72* \a:2.3) -- (-5 + \b -20 + 72* \a:3);
      }
    }

    \foreach \a in {0, 1, ..., 4}{
      \draw[fill] (-20 + 72* \a:1.2) circle (0.13)
      (-20 + 72* \a:2.3) circle (0.13)
      (15-20 + 72* \a:2.2) circle (0.13)
      (-15-20 + 72* \a:2.2) circle (0.13);
    }

    \draw (0,0) node at (-72:1.4) {$v_1$}
    node at (-72*2:1.4) {$v_2$}
    node at (-72*3:1.4) {$v_3$}
    node at (-72*4:1.4) {$v_4$}
    node at (-72*5:1.4) {$v_5$};

  \end{tikzpicture}
  \caption{Colourings resulting in $\cC^*$-graphs and $\cB$-graphs.}
  \label{fig:Cstar-B-colourings}
\end{figure}  

\subsection{Proof of Proposition~\ref{prop:BC-inverse}}
\label{sec:proof-proposition-BC-inverse}

The situation gets a little more complex once we allow brooms and $\cC^*$-graphs to mingle.  As we will see in this subsection, for every family $\cF$ of graphs that contains both $\cB$-graphs and $\cC^*$-graphs, there are graphs $G$ with $m(G) \le 1$ that are list-Ramsey for $\cF$.

\begin{dfn}
  Given a broom $B$ with $b$ hairs, an odd cycle $C$ of length $\ell$, and a star $S$ with $s$ rays, we define the following two graphs:
  \begin{itemize}
  \item
    $F(B, C)$ is obtained from $C$ by attaching a star with $2b-1$ rays to each vertex;
  \item
    $T(B,S)$ is a rooted tree with two layers, where the root has degree $2\cdot \max\{b+1,s\} - 1$ and each of its children has $s$ children.
  \end{itemize}
\end{dfn}

Obviously $m(F(B,C)) = 1$ and $m(T(B,S)) \le 1$. We will show that for any odd cycle and broom, there is a graph of that ilk that is Ramsey for them.
 
\begin{claim}
  \label{claim:broom-cycle}
  For every broom $B$ and odd cycle $C$, the graph $F(B,C)$ is $2$-Ramsey with respect to $\{B, C\}$.
\end{claim}

\begin{proof}
  Let $b$ denote the number of hairs in $B$.  Fix some $\{\text{red}, \text{blue}\}$-colouring of $F(B,C)$.  We may assume that the cycle in $F(B, C)$ is not monochromatic, since otherwise there is nothing left to prove.  Let $P_1, ..., P_{2k}$ be all the maximal monochromatic subpaths of the cycle, labeled in such a way that $P_i$ lies between $P_{i-1}$ and $P_{i+1}$.  By the pigeonhole principle, some $b$ (out of $2b-1$) rays of the star attached to every endpoint of each of these paths are coloured the same; colour each endpoint with such majority colour.  If both endpoints of some $P_i$ have the same colour as $P_i$, then there is a monochromatic copy of $B$, so we may assume that this is not the case.  This means that the colouring must have the following regular form (`right' can be changed to `left'): Each $P_i$ and its right endpoint have the same colour, which depends on the parity of $i$.  However, some $P_i$ must have length at least two (since the cycle is odd), but then $P_i$ and the star at its right endpoint contain a monochromatic copy of $B$.
\end{proof}

\begin{claim}
  \label{claim:broom-star}
  For every broom $B$ and star $S$, any $\{\text{red}, \text{blue}\}$-colouring of the edges of $T(B,S)$ admits either a monochromatic copy of $B$ or copies of $S$ in both colours.
\end{claim}

\begin{proof}
  By the pigeonhole principle, in any $\{\text{red}, \text{blue}\}$-colouring of $T(B,S)$, its root $v$ must be connected to $m \coloneqq \max\{b+1,s\}$ vertices, say $u_1, \dotsc, u_m$, in one of the colours, say red.  We thus already have a red star with $m$ rays, and in particular a red copy of $S$.  If the edge connecting $u_1$ to any of its $s$ children is red, then we have a red copy of the broom with $m-1$ hairs, and in particular a red copy of $B$. Otherwise, all the edges connecting $u_1$ to its children are blue, meaning that we have a blue copy of $S$.
\end{proof}

\begin{proof}[Proof of Proposition~\ref{prop:BC-inverse}]
  Suppose that $\cF$ contains a $\cB$-graph $X$ and a $\cC^*$-graph $Y$.  We will construct a graph $G$ with $m(G) = 1$, with a list of two colours assigned to each edge, such that any colouring of $G$ from these lists yields a monochromatic copy of either $X$ or $Y$.

  Let $B$ be a broom with sufficiently many hairs so that every component of $X$ is a subgraph of $B$ and suppose that the connected components of $Y$ are stars $\{S^1, \dotsc, S^k\}$ and (subgraphs of) odd cycles $\{C^1, \dotsc, C^m\}$.  Define, for all $j \in \br{k + m}$,
  \[
    A_j \coloneqq
    \begin{cases}
      T(B, S^j), & j\le k, \\
      F(B, C^{j-k}), & j > k,
    \end{cases}
  \]    	
  and let $A$ be the disjoint union of all $A_j$ for $j \in \br{k+m}$.

  We first define a graph $G'$, together with an assignment of lists to its edges, in the following way. For every $\{a,b\} \in \binom{\br{k+m+1}}{2}$, add to $G'$ a disjoint copy of $A$ whose every edge receives the same colour list $\{a, b\}$. We claim that every colouring of $G'$ from these lists either has a monochromatic copy of $B$ or a monochromatic copy of $Y$. Indeed, suppose we colour from the lists without creating a monochromatic copy of $B$.  By Claim~\ref{claim:broom-star}, for every $j \le k$ and each pair $\{a,b\}$, the colouring of the corresponding copy of $A_j = T(B, S^j)$ will contain a monochromatic copy of $S^j$ in either colour $a$ or colour $b$.\footnote{Actually, Claim~\ref{claim:broom-star} guarantees that the colouring of $A_j$ contains monochromatic copies of $S^j$ in both colours, but we do not need this stronger statement here.} This means that, for every $j$, there is a copy of $S^j$ in all but at most one colour from $\br{k+m+1}$.  The same argument (using Claim~\ref{claim:broom-cycle}) applies to every $C^j$. Since there are $m+k+1$ colours and $m+k$ graphs, there must be some colour that contains a copy of all $S^j$ and all $C^j$.  Consequently, every colouring of $G'$ from the lists has a monochromatic copy of either $B$ or $Y$.

  Suppose now that $X$ has $t$ connected components. Define $G$ to be the disjoint union of $(t-1)(k+m+1) + 1$ copies of $G'$ together with their list assignments. Any colouring of $G$ from these lists, either has a monochromatic copy of $Y$ in one of the copies of $G'$ or a monochromatic copy of $B$, in some colour in $\br{k+m+1}$, in each copy of $G'$.  In the latter case, by the pigeonhole principle, there is a colour containing $t$ disjoint copies of $B$, thus introducing a monochromatic copy of $X$. 
\end{proof}

\subsection{Proof of Proposition~\ref{prop:BC-non-list}}

Suppose first that $\cA \coloneqq \Aux(\cF)$ is $2$-colourable and that $\cF$ does not contain a star forest. We wish to show that any $G$ with $m(G) \le 1$ admits a $\{\text{red}, \text{blue}\}$-colouring of its edges without a monochromatic copy of any member of $\cF$. Fix a proper $2$-colouring $\varphi \colon V(\cA) \to \{\text{red}, \text{blue}\}$ of the auxiliary hypergraph $\cA$ and denote the connected components of $G$ by $G_1, \dotsc, G_m$. For every $i \in \br{m}$, we colour $G_i$ as follows:
\begin{itemize}
\item
  If $G_i$ contains an odd cycle $C$, colour the edges of the cycle with $\varphi(C)$ and the remaining edges according to the parity of their distance from the cycle, so that the monochromatic components are $C$ and stars.
\item
  Otherwise, if $G_i$ contains no odd cycle, colour it as in Claim~\ref{claim:trivial-colouring}, so that both monochromatic graphs are star forests.
\end{itemize}
We claim that this colouring if $\cF$-free.  Indeed, every monochromatic subgraph must be a $\cC^*$-graph.  Furthermore, if $F \in \cF$ became monochromatic and its non-star components are $F^1, \dotsc, F^m$, there were odd cycles $C^1, \dotsc, C^m$, with $F^i \subseteq C^i$ for each $i$, that all received the same colour.  However, this would mean that the edge $\{C^1, \dotsc, C^m\} \in \cA$ was monochromatic, a contradiction.

Suppose now that either $\cF$ contains a star forest or $\cA$ is not $2$-colourable. If the former holds, and $\cF$ contains a union of $m$ stars, each having at most $s$ rays, then the disjoint union of $2m-1$ copies of a star with $2s-1$ rays, which has density less than one, is clearly Ramsey for $\cF$.  We may therefore suppose that the latter holds, i.e., that $\cA$ is not $2$-colourable.  Using compactness, we can find a finite subhypergraph $\cA' \subseteq \cA$ that is already not $2$-colourable.

Let $X$ be an arbitrary $\cB$-graph from $\cF$ and let $B$ be a broom with sufficiently many hairs so that each component of $X$ is a subgraph of $B$.  Define $G'$ to be a graph whose connected components are $\{F(B,C) : C \in V(\cA')\}$.  We claim that every $\{\text{red}, \text{blue}\}$-colouring of the edges of $G'$ admits a monochromatic copy of either $B$ or the disjoint union $C^1 \cup \dotsb \cup C^k$ for some $\{C^1, \dotsc , C^k\} \in \cA'$.  Indeed, since each component of $G'$ has a monochromatic copy of $B$ or the corresponding cycle, by Claim~\ref{claim:broom-cycle}, we either have a monochromatic copy of $B$ or a monochromatic copy of every $C \in V(\cA')$.  However, since the hypergraph $\cA'$ is not $2$-colourable, there must be some collection $C^1, \dotsc, C^k$ of cycles that form an edge in $\cA'$ and were all coloured the same.

In order to define $G$, consider the subfamily $\cF' \subseteq \cF$ that consists of all $\cC^*$-graphs that contribute an edge to $\cA'$.  Since $\cF'$ is finite, we may give the next definitions:
\begin{itemize}
\item
  Let $S$ be a star with sufficiently many rays so that $S$ contains every star in each $F \in \cF'$.
\item
  For every graph $A$, let $\rho(A)$ denote the maximum, over all $F \in \cF'$, number of components $F^0$ of $F$ such that $F^0 \subseteq A$.
\end{itemize}
Now, let $\rho \coloneqq \max \{ \rho(C) \colon C \in V(\cA') \} $ and let $x$ denote the number of components of~$X$. Finally, define $G$ to be the disjoint union of $2(\rho-1) \cdot e(\cA') + 2x - 1$ copies of $G'$ and $\rho(S) + 2x - 2$ copies of $T(B, S)$.  Since every component of $G$ has at most one cycle, we clearly have $m(G) \le 1$.  To finish the proof, it suffices to show that $G$ is Ramsey for $\cF$.

Consider any $\{\text{red}, \text{blue}\}$-colouring of $G$.  We have already shown that every $G'$-component contains a monochromatic copy of either $B$ or the disjoint union $C^1 \cup \dotsb \cup C^k$ for some $\{C^1, \dotsc, C^k\} \in \cA'$. Moreover, by Claim~\ref{claim:broom-star}, each $T(B,S)$-component contains either a monochromatic copy of $B$ or monochromatic copies of $S$ in both colours.  If there are $2x-1$ monochromatic copies of $B$, some $x$ amongst them have the same colour and therefore we get a monochromatic copy of $X$.  Otherwise, there are at least $2(\rho-1) \cdot e(\cA') + 1$ copies of $G'$ and $\rho(S)$ copies of $T(B, S)$ that do not contain a monochromatic copy of $B$.  This means that we have $\rho(S)$ copies of $S$ in both colours and $2(\rho-1)\cdot e(\cA') + 1$ monochromatic edges of $\cA'$.  By the pigeonhole principle, some edge $\{C^1, \dotsc, C^k \} \in \cA$ is repeated $\rho$ times in the same colour, say red.  Let $F \in \cF'$ be a graph that generated this edge.  We claim that $G$ has a monochromatic copy of $F$.  Indeed, every component of $F$ is either a star, and thus contained in $S$, or contained in some $C^i$ for $i \in \br{k}$.  Moreover, $F$ has at most $\rho(S)$ star components and at most $\rho$ components that are contained in every $C^i$.  We conclude that there is a red copy of $F$.

\bibliographystyle{amsplain}
\bibliography{choosability}

\providecommand{\bysame}{\leavevmode\hbox to3em{\hrulefill}\thinspace}
\providecommand{\MR}{\relax\ifhmode\unskip\space\fi MR }
\providecommand{\MRhref}[2]{%
  \href{http://www.ams.org/mathscinet-getitem?mr=#1}{#2}
}
\providecommand{\href}[2]{#2}
\begin{thebibliography}{10}

\bibitem{AloBucKalKupSza21}
Noga Alon, Matija Buci\'{c}, Tom Kalvari, Eden Kuperwasser, and Tibor
  Szab\'{o}, \emph{List {R}amsey numbers}, J. Graph Theory \textbf{96} (2021),
  no.~1, 109--128.

\bibitem{FriKri00}
Ehud Friedgut and Michael Krivelevich, \emph{Sharp thresholds for certain
  {R}amsey properties of random graphs}, Random Structures \& Algorithms
  \textbf{17} (2000), no.~1, 1--19.

\bibitem{FriKupSamSch22}
Ehud Friedgut, Eden Kuperwasser, Wojciech Samotij, and Mathias Schacht,
  \emph{Sharp thresholds for {R}amsey properties}, arXiv:2207.13982.

\bibitem{GraRodRuc96}
Ronald Graham, Vojt\v{e}ch R\"{o}dl, and Andrzej Ruci\'{n}ski, \emph{On {S}chur
  properties of random subsets of integers}, J. Number Theory \textbf{61}
  (1996), no.~2, 388--408.

\bibitem{Hyde23}
Joseph Hyde, \emph{Towards the 0-statement of the {K}ohayakawa--{K}reuter
  conjecture}, Combinatorics, Probability and Computing \textbf{32} (2023),
  no.~2, 225--268.

\bibitem{JanLucRuc00}
Svante Janson, Tomasz {\L}uczak, and Andrzej Ruci{\'n}ski, \emph{Random
  graphs}, John Wiley \& Sons, 2000.

\bibitem{KohKre97}
Yoshiharu Kohayakawa and Bernd Kreuter, \emph{Threshold functions for
  asymmetric {R}amsey properties involving cycles}, Random Structures \&
  Algorithms \textbf{11} (1997), no.~3, 245--276.

\bibitem{LieMatMenSko20}
Anita Liebenau, Let{\'\i}cia Mattos, Walner Mendon{\c{c}}a, and Jozef Skokan,
  \emph{Asymmetric {R}amsey properties of random graphs involving cliques and
  cycles}, Random Structures \& Algorithms \textbf{62} (2023), no.~4,
  1035--1055.

\bibitem{MarSkoSpoSte09}
Martin Marciniszyn, Jozef Skokan, Reto Sp{\"o}hel, and Angelika Steger,
  \emph{Asymmetric {R}amsey properties of random graphs involving cliques},
  Random Structures \& Algorithms \textbf{34} (2009), no.~4, 419--453.

\bibitem{MouNenSam20}
Frank Mousset, Rajko Nenadov, and Wojciech Samotij, \emph{Towards the
  {K}ohayakawa--{K}reuter conjecture on asymmetric {R}amsey properties},
  Combinatorics, Probability and Computing \textbf{29} (2020), no.~6, 943--955.

\bibitem{NenSte16}
Rajko Nenadov and Angelika Steger, \emph{A short proof of the random {R}amsey
  theorem}, Combin. Probab. Comput. \textbf{25} (2016), no.~1, 130--144.

\bibitem{Ram30}
Frank~P. Ramsey, \emph{On a problem of formal logic}, Proceedings of the London
  Mathematical Society \textbf{2} (1930), no.~1, 264--286.

\bibitem{RodRuc95}
Vojt\v{e}ch R\"{o}dl and Andrzej Ruci\'{n}ski, \emph{Threshold functions for
  {R}amsey properties}, J. Amer. Math. Soc. \textbf{8} (1995), no.~4, 917--942.

\bibitem{RodRuc97}
\bysame, \emph{Rado partition theorem for random subsets of integers}, Proc.
  London Math. Soc. (3) \textbf{74} (1997), no.~3, 481--502.

\bibitem{Sch17}
Issai Schur, \emph{{\"U}ber die {K}ongruenz {$x^{m}+ y^{m} \equiv z^{m}
  \pmod{p}$}}, Jahresbericht der Deutschen Mathematiker-Vereinigung \textbf{25}
  (1917), 114--116.

\bibitem{vdW27}
Bartel~L. {van der Waerden}, \emph{{Beweis einer Baudetschen Vermutung}}, Nieuw
  Arch. Wiskd., II. Ser. \textbf{15} (1927), 212--216 (German).

\end{thebibliography}
\end{document}